\documentclass{wyl}

\numberwithin{equation}{section} 

\begin{document}

 \abovedisplayskip 6pt plus 2pt minus 2pt \belowdisplayskip 6pt
plus 2pt minus 2pt
\def\vsp{\vspace{1mm}}
\def\th#1{\vspace{1mm}\noindent{\bf #1}\quad}
\def\proof{\vspace{1mm}\noindent{\it Proof}\quad}
\def\no{\nonumber}
\newenvironment{prof}[1][Proof]{\noindent\textit{#1}\quad }
{\hfill $\Box$\vspace{0.7mm}}
\def\q{\quad} \def\qq{\qquad}
\allowdisplaybreaks[4]


\AuthorMark{  }                             

\TitleMark{  }  

\title{A class of finite $p$-groups and the normalized unit groups of group algebras     
\footnote{Supported by National Natural Science Foundation of China (Grant No.12171142)}}                  

\author{Yulei Wang}
    {Department of Mathematics, \ Henan University of Technology,  \ Zhengzhou\ 450001, P. R. China\\
    E-mail\,$:$ yulwang@haut.edu.cn}

\author{Heguo  Liu}
    {Department of Mathematics, \ Hainan University,  \ Haikou\ 570228, P. R. China\\
    E-mail\,$:$ ghliu@hainanu.edu.cn}

\maketitle%

\Abstract{Let $p$ be a prime and $\mathbb{F}_p$ be a finite field of $p$ elements. Let $\mathbb{F}_pG$ denote the
group algebra of the finite $p$-group $G$ over the field $\mathbb{F}_p$ and $V(\mathbb{F}_pG)$ denote the  group of normalized units in $\mathbb{F}_pG$.
Suppose that $G$ is a  finite $p$-group given by a
central extension of the form
$$1\longrightarrow \mathbb{Z}_{p^n}\times \mathbb{Z}_{p^m} \longrightarrow G \longrightarrow \mathbb{Z}_p\times
\cdots\times \mathbb{Z}_p \longrightarrow 1$$ and $G'\cong
\mathbb{Z}_p$, $n, m\geq 1$ and $p$ is odd.  In this paper, the structure of $G$ is determined. And
the relations of $V(\mathbb{F}_pG)^{p^l}$ and $G^{p^l}$,
$\Omega_l(V(\mathbb{F}_pG))$ and $\Omega_l(G)$  are given.
Furthermore, there is a direct proof for  $V(\mathbb{F}_pG)^p\bigcap G=G^p$.}    


\MRSubClass{20C05, 20D15}      

\section{Introduction}
 Since the theory of finite solvable groups has been developed to a satisfactory degree, what remains now is the study and classification of finite $p$-groups.
   Classifying groups up to isomorphism is a fundamental problem in group theory.
   But there are not general ways unless one restricts to particularly well behaved groups. The special case of groups of prime power order is particularly difficult  as its was observed by P. Hall  in \cite{Hall}  where he wrote: `` To put is crudely, there is no apparent limit to the complication of a prime power group... And its seems unlikely that it will be possible to compass the overwhelming variety of prime power groups within the bounds of a single finite system of formulae".
  For example, Blackburn classified the finite $p$-groups with derived subgroup of order $p$ by associating to such a finite $p$-group a set of integer vectors having certain properties in \cite{Blackburn}. Moreover, it is also difficult to give in detail the structure of finite $p$-groups. For some special group classes, some results only can be obtained such as  extra-special $p$-groups in \cite{Winter}, generalized extra-special $p$-groups and more generalized $p$-groups in \cite{Dietz, LW, WL2},   $p$-groups with a cyclic Frattini subgroup and odd order $p$-groups of class $2$ such that the quotient by the center is homocyclic in \cite{Bornand}, and so on.

  According to  the properties of central product and inner abelian $p$-groups, we  consider the class of finite $p$-groups with derived subgroup of order $p$  given a central extension of an abelian group with rank $2$ by an  elementary abelian group. If $p=2$, we gave the structures of the class of finite $2$-groups in order to determine the orders of their unitary subgroups of group algebras in \cite{WL3}. In the article, we will consider the case $p$ is odd and study the properties of  the normalized unit groups of group algebras.

  Let $FG$ be the group algebra of a finite group $G$ over a field $F$.
Let $\varepsilon$ be the augmentation mapping of $FG$, which is
the homomorphism $\varepsilon: FG\mapsto F$ given by
$$\varepsilon \left(\sum\limits_{g\in G}\alpha_gg \right)=\sum\limits_{g\in G}\alpha_g.$$
The kernel of $\varepsilon$ is an ideal of $FG$ and the ideal is denoted by $\triangle(G)$.
Let $V(FG)$ be the normalized unit group of the group algebra $FG$ of a finite group $G$ over a field $F$, that is,
$V(FG)=\left\{x\in FG \ |\ \varepsilon(x)=1\right\}=1+\triangle(G).$
 Study of units and their properties is one of the main research problems in group ring theory.  Results obtained in this direction are also useful for the investigation of  central series, $p$-central, automorphisms, Lie properties of group rings and other open questions in this area (see, for example, \cite{Bovdi,Bovdi2}).

 One knows that for a finite abelian $p$-group $A$ and the field $\mathbb{F}_p$ of $p$ elements,   $V(\mathbb{F}_pA)^{p^l}=V(\mathbb{F}_pA^{p^l})$
and $\Omega_l(V(\mathbb{F}_pA))= 1+\triangle(A,\Omega_l(A)) $, $l\geq 1$ (as in \cite{Sandling}).
We will naturally consider the relations of $V(\mathbb{F}_pG)^{p^l}$ and $G^{p^l}$,
$\Omega_l(V(\mathbb{F}_pG))$ and $\Omega_l(G)$  for the non-abelian $p$-group $G$. Moreover,
 there is a question: for which $p$-group $G$ is it true that $G\bigcap V(\mathbb{F}_pG)^p=G^p$ in \cite{Johnson}.
The question was solved affirmatively  for some classes of finite $p$-groups
 which have  the normal complement in the normalized unit groups of group algebras (as in \cite{Sandling2, Johnson}).
 At this time, we may give directly proof for the $p$-groups we will consider.  Our main results are as follows.

\begin{theorem}\label{StrG}
Let  $G$ be a  finite $p$-group given by a
central extension of the form
$$1\longrightarrow N=\mathbb{Z}_{p^n}\times \mathbb{Z}_{p^m} \longrightarrow G \longrightarrow \mathbb{Z}_p\times
\cdots\times \mathbb{Z}_p \longrightarrow 1$$ and $G'=\langle c\rangle\cong
\mathbb{Z}_p$, $n\geq m\geq 1$ and $p>2$. Suppose that  the order of $G/\zeta G$ is  $p^{2k}$ and $d(\zeta G)=r$. Then

(1) $\zeta G=\langle z_1\rangle\times\langle z_2\rangle\times\cdots\times\langle z_r\rangle$ has four isomorphism types $A_i$, $i=1,2,3,4$,

(i) $A_1\cong \mathbb{Z}_{p^n}\times\mathbb{Z}_{p^m}\times\mathbb{Z}_p\times\cdots\times \mathbb{Z}_p$ and $N=\langle z_1\rangle\times\langle z_2\rangle$;

(ii) $A_2\cong \mathbb{Z}_{p^{n+1}}\times\mathbb{Z}_{p^m}\times\mathbb{Z}_p\times\cdots\times \mathbb{Z}_p$ and $N=\langle z_1^p\rangle\times\langle z_2\rangle$;

(iii) $A_3\cong \mathbb{Z}_{p^n}\times\mathbb{Z}_{p^{m+1}}\times\mathbb{Z}_p\times\cdots\times \mathbb{Z}_p$ and $N=\langle z_1\rangle\times\langle z_2^p\rangle$;

(iv) $A_4\cong \mathbb{Z}_{p^{n+1}}\times\mathbb{Z}_{p^{m+1}}\times\mathbb{Z}_p\times\cdots\times \mathbb{Z}_p$ and $N=\langle z_1^p\rangle\times\langle z_2^p\rangle$.

(2) For the isomorphism type $A_1$ and $c\in \langle z_1\rangle$,  the isomorphism types of $G$ are as follows.
\begin{align}\label{1.1}
M_p(n+1,m+1){\mathsf Y}\underbrace{M_p(1,1,1)\Y\cdots \Y M_p(1,1,1)}_{k-1}\times \underbrace{\mathbb{Z}_p\times\cdots\times \mathbb{Z}_p}_{r-2},
\end{align}
\begin{align}\label{1.2}
 M_p(n+1,1)\Y M_p(m+1,1,1)\Y \underbrace{M_p(1,1,1)\Y \cdots \Y M_p(1,1,1)}_{k-2}\times \underbrace{\mathbb{Z}_p\times\cdots\times \mathbb{Z}_p}_{r-2},
\end{align}
\begin{align}\label{1.3}
 M_p(n+1,1)\Y \underbrace{M_p(1,1,1){\mathsf Y}\cdots \Y M_p(1,1,1)}_{k-1}\times \mathbb{Z}_{p^m}\times\underbrace{\mathbb{Z}_p\times\cdots\times \mathbb{Z}_p}_{r-2},
\end{align}
\begin{align}\label{1.4}
 M_p(m+1,1,1)\Y \underbrace{M_p(1,1,1)\Y \cdots \Y M_p(1,1,1)}_{k-1}\Y  \mathbb{Z}_{p^n}\times\underbrace{\mathbb{Z}_p\times\cdots\times \mathbb{Z}_p}_{r-2},
\end{align}
\begin{align}\label{1.5}
 \underbrace{M_p(1,1,1)\Y \cdots \Y M_p(1,1,1)}_{k}\Y  \mathbb{Z}_{p^n}\times\mathbb{Z}_{p^{m}}\times\underbrace{\mathbb{Z}_p\times \cdots\times \mathbb{Z}_p}_{r-2}.
\end{align}
(3) For  the isomorphism type $A_1$ and $c\in \langle z_2\rangle$,  the isomorphism types of $G$ are as follows.
\begin{align}\label{1.6}
M_p(m+1,n+1)\Y \underbrace{M_p(1,1,1)\Y \cdots \Y M_p(1,1,1)}_{k-1}\times \underbrace{\mathbb{Z}_p\times\cdots\times \mathbb{Z}_p}_{r-2},
\end{align}
\begin{align}\label{1.7}
 M_p(m+1,1)\Y M_p(n+1,1,1)\Y \underbrace{M_p(1,1,1)\Y \cdots \Y M_p(1,1,1)}_{k-2}\times \underbrace{\mathbb{Z}_p\times\cdots\times \mathbb{Z}_p}_{r-2},
\end{align}
\begin{align}\label{1.8}
 M_p(m+1,1)\Y \underbrace{M_p(1,1,1)\Y \cdots \Y M_p(1,1,1)}_{k-1}\times \mathbb{Z}_{p^n}\times\underbrace{\mathbb{Z}_p\times\cdots\times \mathbb{Z}_p}_{r-2},
\end{align}
\begin{align}\label{1.9}
 M_p(n+1,1,1)\Y \underbrace{M_p(1,1,1)\Y \cdots \Y M_p(1,1,1)}_{k-1}\Y  \mathbb{Z}_{p^m}\times\underbrace{\mathbb{Z}_p\times\cdots\times \mathbb{Z}_p}_{r-2},
\end{align}
\begin{align}\label{1.10}
 \underbrace{M_p(1,1,1)\Y \cdots \Y M_p(1,1,1)}_{k}\Y  \mathbb{Z}_{p^m}\times\mathbb{Z}_{p^{n}}\times\underbrace{\mathbb{Z}_p\times\cdots\times \mathbb{Z}_p}_{r-2}.
\end{align}
(4) For  the isomorphism type $A_2$ and $c\in \langle z_1\rangle$,  the isomorphism types of $G$ are as follows.
\begin{align}\label{1.11}
 M_p(m+1,1,1)\Y \underbrace{M_p(1,1,1)\Y \cdots \Y M_p(1,1,1)}_{k-1}\Y  \mathbb{Z}_{p^{n+1}}\times\underbrace{\mathbb{Z}_p\times\cdots\times \mathbb{Z}_p}_{r-2},
\end{align}
\begin{align}\label{1.12}
 \underbrace{M_p(1,1,1)\Y \cdots \Y M_p(1,1,1)}_{k}\Y  \mathbb{Z}_{p^{n+1}}\times\mathbb{Z}_{p^{m}}\times\underbrace{\mathbb{Z}_p\times\cdots\times \mathbb{Z}_p}_{r-2}.
\end{align}
(5) For the isomorphism type $A_2$ and $c\in \langle z_2\rangle$,  the isomorphism types of $G$ are as follows.
\begin{align}\label{1.13}
 M_p(m+1,1)\Y \underbrace{M_p(1,1,1)\Y \cdots \Y M_p(1,1,1)}_{k-1}\times \mathbb{Z}_{p^{n+1}}\times\underbrace{\mathbb{Z}_p\times\cdots\times \mathbb{Z}_p}_{r-2},
\end{align}
\begin{align}\label{1.14}
 \underbrace{M_p(1,1,1)\Y \cdots \Y M_p(1,1,1)}_{k}\Y  \mathbb{Z}_{p^m}\times\mathbb{Z}_{p^{n+1}}\times\underbrace{\mathbb{Z}_p\times\cdots\times \mathbb{Z}_p}_{r-2}.
\end{align}
(6) For the isomorphism type $A_3$ and $c\in \langle z_1\rangle$,  the isomorphism types of $G$ are as follows.
\begin{align}\label{1.15}
 M_p(n+1,1)\Y \underbrace{M_p(1,1,1)\Y \cdots \Y M_p(1,1,1)}_{k-1}\times \mathbb{Z}_{p^{m+1}}\times\underbrace{\mathbb{Z}_p\times\cdots\times \mathbb{Z}_p}_{r-2},
\end{align}
\begin{align}\label{1.16}
 \underbrace{M_p(1,1,1)\Y \cdots \Y M_p(1,1,1)}_{k}\Y  \mathbb{Z}_{p^n}\times\mathbb{Z}_{p^{m+1}}\times\underbrace{\mathbb{Z}_p\times\cdots\times \mathbb{Z}_p}_{r-2}.
\end{align}
(7) For the isomorphism type $A_3$ and $c\in \langle z_2\rangle$,  the isomorphism types of $G$ are as follows.
\begin{align}\label{1.17}
 M_p(n+1,1,1)\Y \underbrace{M_p(1,1,1)\Y \cdots \Y M_p(1,1,1)}_{k-1}\Y  \mathbb{Z}_{p^{m+1}}\times\underbrace{\mathbb{Z}_p\times\cdots\times \mathbb{Z}_p}_{r-2},
\end{align}
\begin{align}\label{1.18}
 \underbrace{M_p(1,1,1)\Y \cdots \Y M_p(1,1,1)}_{k}\Y  \mathbb{Z}_{p^{m+1}}\times\mathbb{Z}_{p^{n}}\times\underbrace{\mathbb{Z}_p\times\cdots\times \mathbb{Z}_p}_{r-2}.
\end{align}
(8) For the isomorphism type $A_4$,  the isomorphism type of $G$ is
\begin{align}\label{1.19}
\underbrace{M_p(1,1,1)\Y \cdots \Y M_p(1,1,1)}_{k}\Y  (\mathbb{Z}_{p^{n+1}}\times\mathbb{Z}_{p^{m+1}})\times\underbrace{\mathbb{Z}_p\times\cdots\times \mathbb{Z}_p}_{r-2}.
\end{align}
\end{theorem}

\begin{theorem} Suppose that $G$ is the finite $p$-group as in Theorem 1.1 and $C$ is defined in Section 4. Let $\mathbb{F}_p$ be the field of $p$ elements, then

(1) $V(\mathbb{F}_pG)^{p^l}=V(\mathbb{F}_pG^{p^l})$, where $l\geq 2$,
$V(\mathbb{F}_pG)^{p}C=V(\mathbb{F}_pG^{p})\times C$.

(2) $G\bigcap V(\mathbb{F}_pG)^{p}=G^p$.

(3) $\Omega_1(V(\mathbb{F}_pG))\leq 1+\triangle(G,\Omega_1(G))$ and
$\Omega_l(V(\mathbb{F}_pG))= 1+\triangle(G,\Omega_l(G))$, where $l\geq 2.$

(4) $\zeta \Omega_l(V(\mathbb{F}_pG))=\Omega_l(\zeta V(\mathbb{F}_pG))=\Omega_l(V(\mathbb{F}_p\zeta G))\times C$, where $l\geq 1$.

\end{theorem}

\section{Preliminaries}

To avoid confusion, we will explain some notations.
$G$ is a central product of its subgroups $H$ and $K$ if $G=HK$, $[H,K]=1$,  we write $ G=H \Y K$.
For $x=\sum\limits_{g\in G}\alpha_gg\in FG$, let $\mathrm{supp}(x)$ denote the support of $x\in FG$, namely,
$\mathrm{supp}(x)=\{g\in G\ |\ \alpha_g\neq 0\}.$
For any subset $S$ of $G$, we define $\widehat{S}:=\sum\limits_{g\in S}g$. We denote by $d(G)$ the number of elements in the minimal sets
of generators of $G$. Other notations used are standard (as in \cite{XuQ,Robinson}).

\begin{lemma}[\cite{XuQ}] \label{INA}
Suppose that  $G$ is  a finite $p$-group, then the following properties are equivalent:

(i) $G$ is an inner abelian group;

(ii) $d(G)=2$ and $|G'|=p$;

(iii) $d(G)=2$ and $\zeta G=\mathrm{Frat} (G)$.

\end{lemma}

\begin{lemma}[\cite{XuQ}] \label{INA1}
Suppose that  $G$ is an inner abelian finite $p$-group, then $G$ is one of the following types:

(i) $Q_8$, the quaternion group;

(ii) $M_p(n,m)=\langle a,b\ |\ a^{p^n}=b^{p^m}=1,a^b=a^{1+p^{n-1}}\rangle$,  $n\geq 2, m\geq 1$;

(iii) $M_p(n,m,1)=\langle a,b,c\ |\ a^{p^n}=b^{p^m}=c^p=1,[a,b]=c,[c,a]=[c,b]=1\rangle$,  $n\geq  m\geq 1$.

\end{lemma}

For any ring $R$, we denote by $L_n(R)$ the $n$th Lie ideal of
$R$ generated by all the Lie elements $(r_1,\ldots,r_n)$, $r_i\in R$, where $(r_1)=r_1, (r_1,r_2)=r_1r_2-r_2r_1$ and,
inductively, $$(r_1,\ldots,r_{n-1},r_n)=((r_1,\ldots,r_{n-1}),r_n).$$
The $n$th Lie ideal $L_n(R)$ of $R$ and the $n$th term $\gamma_n(U(R))$ of the lower central series of the unit group $U(R)$ of $R$  have very close relationship as follows.

\begin{lemma}[\cite{Gupta}] Suppose that  $R$ is  a ring with identity 1 and $U(R)$ is the group of units of $R$. For each $n\geq1$, $\gamma_n(U(R))\leq 1+L_n(R)$.
\end{lemma}

\begin{lemma}[\cite{Brauer}] \label{Brauer}

Suppose $RG$ is the group ring of a group $G$ over the commutative ring $R$ and  $C_g$ is the conjugacy class of $G$ containing $g$.
Let $[RG,RG]$ be an $R$-submodule of $RG$ generated by all elements of the form $(x,y):=xy-yx$ with $x,y\in RG$ and set $[RG,RG]_p:=[RG,RG]+pRG$,
where $p$ is a prime number. The following conditions hold:

(i) $\sum\limits_{g\in G}\alpha_gg$ is in $[RG,RG]$ if and only if $\sum\limits_{g\in C_h}\alpha_g=0$ for all $h\in G$;

(ii) $\sum\limits_{g\in G}\alpha_gg$ is in $[RG,RG]_p$ if and only if $\sum\limits_{g\in C_h}\alpha_g\in pR$ for all $h\in G$;

(iii) $(x+y)^{p}\equiv x^{p}+y^{p} \pmod{[RG,RG]_p}$ for any $x,y\in RG$;

(iv) the $R$-module $[RG,RG]_p$ contains the $p$-th power of all of its elements.

\end{lemma}

\begin{definition}
Suppose that $R$ is a ring. For a subgroup $H$ of $G$, we  denote by $\triangle_R(G,H)$ the
left ideal of $RG$ generated by the set $\{ h-1\ |\ h\in H\}$, that is,
$$\triangle_R(G,H)=\left\{ \sum\limits_{h\in H}\alpha_h(h-1)\ |\ \alpha_h\in RG\right\}.$$
\end{definition}
 While working with a fixed ring $R$ we shall omit the subscript and denote this ideal simply by
 $\triangle(G,H)$. Obviously, $\triangle(G,H)$ is an ideal of $RG$ if $H\unlhd G$. Notice that the ideal
$\triangle(G,G)$ coincides with the kernel $\triangle(G)$  of the augmentation mapping of the group ring $RG$.

To give a better description of $\triangle_R(G,H)$, let us denote by $\mathcal{T}=\{q_1,q_2,\ldots,q_t\}$, $t=|G/H|$,  a complete
set of representatives of left cosets of  $H$ in $G$. Then the basis of  $\triangle_R(G,H)$  is as follows.

\begin{lemma}[\cite{Milies}]\label{subbas}
The set $\{ q(h-1)\ |\ q\in \mathcal{T},h\in H,h\neq 1\}$ is a basis of $\triangle_R(G,H)$ over $R$.
\end{lemma}

\begin{lemma}[\cite{BB1}] \label{Vpc}
Let $G$ be a finite $p$-group with commutator subgroup $G'$ of order $p>2$. Then
$V(\mathbb{F}_pG)^p$ is a subgroup of the center $\zeta(V(\mathbb{F}_pG))$.
\end{lemma}

\begin{lemma}[\cite{BB1}]\label{abp}
Suppose that $H$ is a group generated by two elements $a, b$, and its commutator subgroup $H'$
is central of prime order $p$. In any group ring of $H$,
$$(a+b)^p=a^p+b^p+\sum\limits_{r=1}^{p-1}\frac{1}{p}\left(
                                                      \begin{array}{c}
                                                        p \\
                                                        r \\
                                                      \end{array}
                                                    \right)
a^rb^{p-r}\widehat{H'}.$$
\end{lemma}

\section{ The structure of  the group}

Supposet that  $G$ is  a nonabelian  $p$-group given by a
central extension of the form
$$1\longrightarrow N \longrightarrow G \longrightarrow \mathbb{Z}_p\times
\cdots\times \mathbb{Z}_p \longrightarrow 1$$ and $G'=\langle c\rangle\cong
\mathbb{Z}_p$, $N\cong \mathbb{Z}_{p^n}\times \mathbb{Z}_{p^m}$,   $n, m\geq 1$.
Without loss of generality, we  suppose $n\geq m\geq1$  in the following sections.

 Note that $G'\leq\mathrm{Frat}\ G\leq N\leq \zeta G$.
Hence $G/\zeta G$ is an elementary abelian $p$-group. Further, we have the following lemma.

\begin{lemma} \label{STG}
 (1) For any two elements
$\bar{x}=x\zeta G$ and $\bar{y}=y\zeta G$ of $G/\zeta G$, write
$[x,y]=c^{r}$ ($0\leq r\leq p$) and  $f(\bar{x},\bar{y})=r$, then
$G/\zeta G$ is a nondegenerate symplectic space over a field $\mathbb{F}_p$.

(2) $G=E\Y\zeta G$, where $E$ is  a central product of  some inner abelian groups.
Furthermore, these inner abelian groups have the types:
 $M_p(u,v)$ and $M_p(w,1,1)$, where $u,v,w\leq n+1$.

\end{lemma}

\begin{proof}
(1) Obviously  $f$ is well-defined. For  $x$, $y$, $x_i$, $y_i\in
G$, $i=1,2$,   $[x_1x_2,y]=[x_1,y][x_2,y]$ and
$[x,y_1y_2]=[x,y_1][x,y_2]$, thus $f$ is bilinear. Since $[x,x]=1$
and $[x,y]=[y,x]^{-1}$, $f(\bar{x},\bar{x})=0$ and
$f(\bar{x},\bar{y})=-f(\bar{y},\bar{x})$,  $G/\zeta G$ is a
symplectic
 space over $\mathbb{F}_p$. If $f(\bar{x},\bar{y})=0$ for all $y\in
G$, then $[x,y]=1$,  which implies that  $x\in \zeta G$. In a word,  the
symplectic space $G/\zeta G$ is nondegenerate.

(2) According to (1), suppose that  the dimension of symplectic
space $G/\zeta G$ is $2k$, and
$\{\bar{x}_1,\ldots, \bar{x}_k,\bar{y}_1,\ldots,\bar{y}_{k}\}$ is a
basis of $G/\zeta G$,  where $\bar{x}_i=x_i\zeta G$ and  $\bar{y}_i=y_i\zeta G$ for $i, j=1, 2, \ldots, k$, satisfying:
 $f(\bar{x}_i,\bar{y}_i)=1$;
$f(\bar{x}_i,\bar{y}_j)=0$ if $(i\neq j)$;
$f(\bar{x}_i,\bar{x}_j)=f(\bar{y}_i,\bar{y}_j)=0$.

Let $G_i:=\langle x_i,y_i\rangle$, then for $i\neq j$,  $[G_i,G_j]=1$. Obviously $d( G_i)=2$  and $|G_i'|=p$, also by Lemma \ref{INA},
$G_i$ is inner abelian. Note that $G_i^p\leq N$.  According to Lemma \ref{INA1},
the isomorphism types of $(G_i)'s$ are  $M_p(u,v)$ and $M_p(w,1,1)$, where $u,v,w\leq n+1$.$\hfill\square$

\end{proof}

Let $r:=d(\zeta G)$ and
\begin{align}\label{3.1}
\zeta G=\langle z_1\rangle\times\langle z_2\rangle\times\cdots\times\langle z_r\rangle\cong\mathbb{Z}_{p^{n_1}}\times \mathbb{Z}_{p^{n_2}}\times\cdots\times\mathbb{Z}_{p^{n_r}},
 n_1\geq n_2\geq\cdots\geq n_r\geq 1.
\end{align}
Since  $(\zeta G)^p\leq N$,   $n_i=1$ for $3\leq i\leq r$ and $\langle z_1^p\rangle\times\langle z_2^p\rangle=\mathrm{Frat \zeta G}\leq N$.
If $n_2\geq 2$, then
\begin{align}\label{3.2}
N\bigcap \left(\langle z_3\rangle\times\cdots\times\langle z_r\rangle\right)=1,
\end{align}
otherwise, $|\Omega_1(N)|\geq p^3$, a contradiction.
If $n_2=1$,  then  $m$ must be $1$ by $N\leq \zeta G$.  At this time,  by adjusting the parameters $z_1,z_2,\ldots,z_r$,  the equation \eqref{3.2} can similarly be obtained.

\begin{lemma} \label{SCG}
(1) $n\leq n_1\leq n+1$ and $m\leq n_2\leq m+1$.

(2) The isomorphism types of  $\zeta G$ are as follows:

 (i) $A_1\cong\mathbb{Z}_{p^{n}}\times \mathbb{Z}_{p^{m}}\times\underbrace{\mathbb{Z}_p\times\cdots\times\mathbb{Z}_p}_{r-2}$.
 In this case, $N=\langle z_1\rangle\times\langle z_2\rangle $;

(ii) $A_2\cong\mathbb{Z}_{p^{n+1}}\times \mathbb{Z}_{p^{m}}\times\underbrace{\mathbb{Z}_p\times\cdots\times\mathbb{Z}_p}_{r-2}$.
 In this case, $N=\langle z_1^p\rangle\times\langle z_2\rangle $;

(iii) $A_3\cong\mathbb{Z}_{p^{n}}\times \mathbb{Z}_{p^{m+1}}\times\underbrace{\mathbb{Z}_p\times\cdots\times\mathbb{Z}_p}_{r-2}$.
 In this case, $N=\langle z_1\rangle\times\langle z_2^p\rangle $;

(iv) $A_4\cong\mathbb{Z}_{p^{n+1}}\times \mathbb{Z}_{p^{m+1}}\times\underbrace{\mathbb{Z}_p\times\cdots\times\mathbb{Z}_p}_{r-2}$.
 In this case, $N=\langle z_1^p\rangle\times\langle z_2^p\rangle $.
\end{lemma}

\begin{proof}
(1) Note that  $N\leq \zeta G$, thus $p^n=\mathrm{Exp}N\leq \mathrm{Exp}(\zeta G)=p^{n_1}$. Obviously  $(\zeta G)^p\leq N$,
we have $n_1-1\leq n$. It follows that  $n\leq n_1\leq n+1$.

If $n_2\leq m-1$, then $N\leq \zeta G$ implies that
 $$\mathbb{Z}_{p^{n-m+1}}\times\mathbb{Z}_{p}\cong N^{p^{m-1}}\leq (\zeta G)^{p^{m-1}}\cong \mathbb{Z}_{p^{n_1-m+1}},$$
 a contradiction. Thus we have  $n_2\geq m$.

Suppose that $n_1=n+1$.
Obviously $\mathbb{Z}_{p^{n_1-1}}\times\mathbb{Z}_{p^{n_2-1}}\cong (\zeta G)^p\leq N$.  If $n_2\geq m+2$, then
$|N|\geq p^{n+m+1}$, a contradiction. Thus we have $m\leq n_2\leq m+1$.
 Suppose that $n_1=n$.
 If $n_2=m+2$, then
 $$(\zeta G)^p\cong\mathbb{Z}_{p^{n-1}}\times\mathbb{Z}_{p^{m+1}}\cong N\cong\mathbb{Z}_{p^n}\times\mathbb{Z}_{p^m}.$$
 From this, we have  $n-1=m$ and $n_1=n=m+1<n_2$, a contradiction.
If $n_2\geq m+3$, then $p^{n+m+1}\leq|(\zeta G)^p|\leq |N|=p^{n+m}$, a contradiction.

(2) Let $N=\langle a\rangle\times\langle b\rangle$, where $|a|=p^n$ and $|b|=p^m$.

(i)  If  $n_1=n$ and $n_2=m$, then according to \eqref{3.2}, we obtain $\zeta G=\langle a,b,z_3,\ldots,z_r\rangle$, which is the isomorphism type $A_1$ of $\zeta G$.

(ii) Suppose that  $n_1=n+1$ and $n_2=m$. We may let $z_1^p=a^ib^j$ since $\langle z_1^p\rangle\leq N$.
If $n>m$, then $i$ and $p$ are coprime  since $|z_1^p|=p^n=|a|>|b|$. In this case,
$N=\langle a^ib^j\rangle\times\langle b\rangle= \langle z_1^p\rangle\times\langle b\rangle$ and $\zeta G=\langle z_1,b,z_3,\ldots,z_r\rangle$ by \eqref{3.2}, which is  the isomorphism type  $A_2$ of  $\zeta G$.

Suppose that $n=m$. Obviously,  one of $i$ and $j$ must be coprime to $p$. If $i$ is coprime to $p$, then
 we may obtain the isomorphism type  $A_2$ of  $\zeta G$ which is similar to the case $n>m$.
 If $j$ is coprime to $p$, then
 $N=\langle a^ib^j\rangle\times\langle a\rangle= \langle z_1^p\rangle\times\langle a\rangle$ and $\zeta G=\langle z_1,a,z_3,\ldots,z_r\rangle$ by \eqref{3.2} which is the isomorphism type  $A_2$ of  $\zeta G$.

(iii) Suppose that  $n_1=n$ and $n_2=m+1$. Not to cause confusion, we may similarly  let $z_1^p=a^ib^j$ and $z_2^p=a^ub^v$.
Obviously, $i$ must be divisible by $p$.  Let $i=pi_1$.

If $v$ is coprime to $p$, then $N=\langle a\rangle\times\langle z_2^p\rangle $ and $\zeta G=\langle a,z_2,z_3,\ldots,z_r\rangle$ by \eqref{3.2},
which is the isomorphism type $A_3$ of $\zeta G$.

Suppose that $v$ is divisible by $p$ and  $v=pv_1$. Obviously $z_2^{p^m}=a^{p^{m-1}u}$ has order $p$.
Since $n=n_1\geq n_2=m+1$, we have $n>m$.
If $j$ is divisible by $p$, then
$z_1^{p^{n-1}}=a^{p^{n-1}i_1}$  is of  order $p$,  a contradiction.
Thus  $(j,p)=1$.  Further  we may obtain  that $n=m+1$, otherwise, $\langle z_1\rangle \bigcap \langle z_2\rangle=\langle a^{p^{m-1}u} \rangle\neq 1$,
a contradiction.
 From this,  $(a^ib^j)^{p^{m-1}}$ is an element with order $p$ of $\langle z_1\rangle$.
 It follows that  $N=\langle a\rangle\times\langle a^ib^j\rangle=\langle a\rangle\times\langle z_1^p\rangle$ and $\zeta G=\langle a,z_1,z_3,\ldots,z_r\rangle$, which is the isomorphism type $A_3$ of $\zeta G$ by adjusting the parameters $z_1$ and $z_2$.

(iv) Suppose that $n_1=n+1$ and $n_2=m+1$. Since $(\zeta G)^p=\langle z_1^p\rangle\times\langle z_2^p\rangle$ has order $p^{n+m}$,
we have $N=(\zeta G)^p$. In this case, we may take $a=z_1^p$ and $b=z_2^{p}$, which is the isomorphism type $A_4$ of $\zeta G$. $\hfill\square$

\end{proof}

According to Lemma \ref{STG}, we know that $G$ is the central product of $E$ and $\zeta G$. Further, suppose that $E$ is the central product of $G_1,G_2,\ldots,G_k$, where the isomorphism types of $(G_i)'s$ are  $M_p(u,v)$ and $M_p(w,1,1)$, where $u,v,w\leq n+1$  as in Lemma  \ref{STG}.
Next we will determine the structure of $G$ according to the types of $\zeta G$ in Lemma \ref{SCG}.

 \subsection{ The isomorphism type $A_1$ of $\zeta G$}

In the section, let
$$\zeta G=\langle z_1\rangle\times\langle z_2\rangle\times\cdots\times\langle z_r\rangle\cong\mathbb{Z}_{p^{n}}\times \mathbb{Z}_{p^{m}}\times\mathbb{Z}_p\times\cdots\times\mathbb{Z}_p, N=\langle z_1\rangle\times\langle z_2\rangle.$$
Obviously, $c\in \Omega_1(N)=\langle z_1^{p^{n-1}}\rangle\times\langle z_2^{p^{m-1}}\rangle$. If $c=z_1^{p^{n-1}}\cdot z_2^{p^{m-1}}$,
 then we may rewrite $N=\langle z_1\rangle\times\langle z_1^{p^{n-m}}z_2\rangle$. At this time, $c=(z_1^{p^{n-m}}z_2)^{p^{m-1}}$.
Without loss of generality, we may always assume that $c\in \langle z_1\rangle$ or $c\in \langle z_2\rangle$.
For every factor $G_i$ of the central product of $E$,  we  will  determine the types of $G_i\Y N$.

\begin{lemma} \label{CPC1}
Suppose $n\geq m\geq 1$ and $c\in \langle z_1\rangle$.

(i) If $G_i\cong M_p(n+1,m+1)$ , then $G_i\Y N=G_i$.  $G_i\ncong M_p(m+1,n+1)$ if $n>m$.

(ii) If $G_i\cong M_p(n+1,v)$, then $1\leq v\leq m+1$.   $G_i\Y N\cong M_p(n+1,1)\times \mathbb{Z}_{p^m}$ if $v\leq m$.

(iii) If $G_i\cong M_p(u,n+1)$, then $2\leq u\leq m+1$ and $n=m$, then  $G_i\Y N\cong M_p(m+1,1,1)\Y\mathbb{Z}_{p^m}$.

(iv) If $G_i\cong M_p(u,v)$, where $2\leq u<n+1$ and $1\leq v<n+1$, then $G_i\Y N\cong M_p(m+1,1,1)\Y \mathbb{Z}_{p^n}$ or $M_p(1,1,1)\Y \mathbb{Z}_{p^{n}}\times \mathbb{Z}_{p^{m}}$.

(v) If $G_i\cong M_p(w,1,1)$, where $1\leq w\leq n+1$, then $G_i\Y N\cong M_p(m+1,1,1)\Y \mathbb{Z}_{p^n}$ or $M_p(1,1,1)\Y \mathbb{Z}_{p^{n}}\times \mathbb{Z}_{p^{m}}$.

\end{lemma}

\begin{proof}
(i) Suppose that
$$G_i=\langle x_i,y_i \ | \ x_i^{p^{n+1}}=y_i^{p^{m+1}}=1, x_i^{y_i}=x_i^{1+p^n} \rangle \cong M_p(n+1,m+1).$$
Obviously  $\zeta G_i=\langle x_i^p\rangle\times\langle y_i^p\rangle\cong \mathbb{Z}_{p^{n}}\times\mathbb{Z}_{p^{m}}$,
thus $N=\zeta G_i$ and  $G_i\Y N=G_i$.

Suppose that
$$G_i=\langle x_i,y_i \ | \ x_i^{p^{m+1}}=y_i^{p^{n+1}}=1, x_i^{y_i}=x_i^{1+p^m} \rangle \cong M_p(m+1,n+1).$$
Obviously  $G_i'=\langle x_i^{p^m}\rangle$. Thus  $\langle x_i^{p^m}\rangle=\langle c\rangle=\langle z_1^{p^{n-1}}\rangle$.
Note that  $\zeta G_i=\langle x_i^p\rangle\times\langle y_i^p\rangle\leq N$,
thus $N=\zeta G_i$. But, the case $n>m$ implies that
$$\langle c\rangle=N^{p^{n-1}}=(\zeta G_i)^{p^{n-1}}=\langle y_i^{p^{n}}\rangle,$$  which is impossible. Hence $G_i\ncong M_p(m+1,n+1)$.

(ii) Suppose that
$$G_i=\langle x_i,y_i \ | \ x_i^{p^{n+1}}=y_i^{p^{v}}=1, x_i^{y_i}=x_i^{1+p^n} \rangle \cong M_p(n+1,v).$$
Obviously,  $G_i'=\langle x_i^{p^n}\rangle=\langle c\rangle=\langle z_1^{p^{n-1}}\rangle$.
Note that $\zeta G_i=\langle x_i^p\rangle\times\langle y_i^p\rangle\leq N$, thus $p^{n+v-1}\leq |N|=p^{n+m}$, which implies that
$v-1\leq m$.

Suppose that $v\leq m$. Let $ x_i^p=z_1^{l}z_2^{j}$ and $ y_i^p=z_1^{s}z_2^{t}$, where
$1\leq l,s\leq p^n, 1\leq j,t\leq p^m$. If $l$ is divisible by $p$, then there exists an integer $k$ and $(k,p)=1$ such that
 $$c^k=x_i^{p^n}=z_1^{p^{n-1}l}z_2^{p^{n-1}j}=z_2^{p^{n-1}j},$$
  which is impossible. From this, $l$ is coprime to $p$.  Since $1=y_i^{p^v}=z_1^{p^{v-1}s}z_2^{p^{v-1}t}$,
$s$ and $t$ are divisible by $p^{n-v+1}$ and $p^{m-v+1}$, respectively. Let $s=ps_1$ and $t=pt_1$.
Obviously, $y_iz_1^{-s_1}z_2^{-t_1}$ has order $p$. At this time,  we may replace $y_i$ by $y_iz_1^{-s_1}z_2^{-t_1}$.
$$ G_i\Y N=\langle x_i,y_iz_1^{-s_1}z_2^{-t_1},z_2 \rangle=\langle x_i,y_iz_1^{-s_1}z_2^{-t_1}\rangle\times\langle z_2 \rangle\cong M_p(n+1,1)\times\mathbb{Z}_{p^m}.$$

(iii) Suppose that
$$G_i=\langle x_i,y_i \ | \ x_i^{p^{u}}=y_i^{p^{n+1}}=1, x_i^{y_i}=x_i^{1+p^{u-1}} \rangle \cong M_p(u,n+1).$$
Obviously $\zeta G_i=\langle x_i^p\rangle\times\langle y_i^p\rangle\leq N$. Hence $p^{u-1+n}\leq |N|=p^{n+m}$. From this, we have
$u-1\leq m$.  Let $ x_i^p=z_1^{l}z_2^{j}$ and $ y_i^p=z_1^{s}z_2^{t}$, where
$1\leq l,s\leq p^n, 1\leq j,t\leq p^m$.

If $n>m$,  then
$$1\neq y_i^{p^n}=z_1^{p^{n-1}s}z_2^{p^{n-1}t}=z_1^{p^{n-1}s}\in \langle c\rangle=\langle x_i^{p^{u-1}}\rangle,$$
which is impossible.
It follows  $n=m$.
Suppose that $u< m+1$.  Note that $1=x_i^{p^u}=z_1^{p^{u-1}l}z_2^{p^{u-1}j}$. Hence both
$l$ and $j$ are divisible by $p^{m-u+1}$. Let $l=pl_1$ and $j=pj_1$.
If $t$ is divisible by $p$, then
$$1\neq y_i^{p^n}=z_1^{p^{n-1}s}z_2^{p^{n-1}t}=z_1^{p^{n-1}s}\in \langle c\rangle,$$ which is a contradiction.
 Hence  $t$ is coprime to $p$. In a word,  we may replace $x_i$ by $x_iz_1^{-l_1}z_2^{-j_1}$ and
$$ G_i\Y N=\langle x_iz_1^{-l_1}z_2^{-j_1},y_i,z_1 \rangle=\langle x_iz_1^{-l_1}z_2^{-j_1},y_i\rangle\Y\langle z_1 \rangle\cong M_p(m+1,1,1)\Y\mathbb{Z}_{p^m}.$$

(iv)
Suppose that
$$G_i=\langle x_i,y_i \ | \ x_i^{p^{u}}=y_i^{p^{v}}=1, x_i^{y_i}=x_i^{1+p^{u-1}} \rangle \cong M_p(u,v),$$
where $u,v\leq n$.
 Let $ x_i^p=z_1^{l}z_2^{j}$ and $ y_i^p=z_1^{s}z_2^{t}$, where
$1\leq l,s\leq p^n, 1\leq j,t\leq p^m$.
Obviously, both $l$ and $s$ are divisible by $p$. Let $l=pl_1$ and $s=ps_1$.

If $(j,p)=1=(t,p)$, then there exist $j_1$ and $t_1$ such that $jj_1\equiv 1\pmod{p^m}$ and $tt_1\equiv 1\pmod{p^m}$.
Thus $x_i^{pj_1}=z_1^{pl_1j_1}z_2$ and $y_i^{pt_1}=z_1^{ps_1t_1}z_2$, which implies that
$(x_i^{j_1}y_i^{-t_1}z_1^{s_1t_1-l_1j_1})^p=1$. From this, we have
$$G_i\Y N=\langle y_iz_1^{-s_1}, x_i^{j_1}y_i^{-t_1}z_1^{s_1t_1-l_1j_1}, z_1\rangle\cong M_p(m+1,1,1)\Y \mathbb{Z}_{p^n}.$$

Suppose that  $(j,p)=1$ and $p|t$. Let $t=pt_2$. Obviously,  the orders of $x_iz_1^{-l_1}$ and $y_iz_1^{-s_1}z_2^{-t_2}$ are $p^{m+1}$ and $p$,  respectively. It follows that
$$G_i\Y N=\langle x_iz_1^{-l_1}, y_iz_1^{-s_1}z_2^{-t_2},z_1\rangle\cong M_p(m+1,1,1)\Y \mathbb{Z}_{p^n}.$$
Similarly, we may obtain the same result for the case  $p|j$ and $(t,p)=1$.

Suppose that  $p|j$ and $p|t$. Let $j=pj_2$ and $t=pt_3$.  It follows that
$$G_i\Y N=\langle x_iz_1^{-l_1}z_2^{-j_2}, y_iz_1^{-s_1}z_2^{-t_3},z_1,z_2\rangle\cong M_p(1,1,1)\Y \mathbb{Z}_{p^n}\times \mathbb{Z}_{p^m}. $$

(v)
Suppose that
$$G_i=\langle x_i,y_i ,c\ | \ x_i^{p^{w}}=y_i^p=c^p=1, [x_i,y_i]=c,[x_i,c]=[y_i,c]=1 \rangle \cong M_p(w,1,1),$$
 where $1\leq w\leq n+1$. Let $ x_i^p=z_1^{l}z_2^{j}$, where
$1\leq l\leq p^n, 1\leq j\leq p^m$.

We  first assume that $w=n+1$. If $n>m$, then $x_i^{p^n}=z_1^{p^{n-1}l}z_2^{p^{n-1}j}=z_1^{p^{n-1}l}$,  a contradiction.
Hence $n=m$. At this time,  if  both $l$ and $j$ can be divisible by $p$, then the order of $z_1^{l}z_2^{j}$ is less than $p^m$, which
is a contradiction. If $j$ can be divisible by $p$ and $(l,p)=1$, then  $x_i^{p^m}=z_1^{p^{m-1}l}z_2^{p^{m-1}j}\in \langle c\rangle$,
which is also a contradiction.  Hence $(j,p)=1$. From this, we have
$$G_i\Y N=\langle x_i, y_i,z_1\rangle\cong M_p(m+1,1,1)\Y \mathbb{Z}_{p^m}.$$

We next suppose that $w\leq n$.
Obviously,  $1=x_i^{p^w}=z_1^{p^{w-1}l}z_2^{p^{w-1}j}$. Thus $l$ can be divisible by $p$. Let $l=pl_1$.
If $(j,p)=1$, then
$$G_i\Y N=\langle x_iz_1^{-l_1}, y_i,z_1\rangle\cong M_p(m+1,1,1)\Y \mathbb{Z}_{p^n}.$$
Suppose $j$ can be divisible by $p$. Let $j=pj_1$. From this, we have
$$G_i\Y N=\langle x_iz^{-l_1}z_2^{-j_1}, y_i,z_1,z_2\rangle\cong M_p(1,1,1)\Y \mathbb{Z}_{p^n}\times\mathbb{Z}_{p^m}.$$ $\hfill\square$

\end{proof}

\begin{lemma} \label{CPC2}
Suppose $n\geq m\geq 1$ and $c\in \langle z_2\rangle$.

(i) If $G_i\cong M_p(m+1,n+1)$ , then $G_i\Y N=G_i$.  $G_i\ncong M_p(n+1,m+1)$ if $n>m$.

(ii) If $G_i\cong M_p(n+1,v)$, where $1\leq v\leq m+1$, then $n=m$.  $G_i\Y N\cong M_p(m+1,1)\times \mathbb{Z}_{p^m}$ if $1\leq v\leq m$.

(iii) If $G_i\cong M_p(u,n+1)$, where $2\leq u\leq m$, then $G_i\Y N\cong M_p(n+1,1,1)\Y\mathbb{Z}_{p^m}$.

(iv) If $G_i\cong M_p(u,v)$, where $2\leq u<n+1$ and $1\leq v<n+1$, then
 $G_i\Y N\cong M_p(m+1,1)\times\mathbb{Z}_{p^n}$ or $M_p(1,1,1)\Y \mathbb{Z}_{p^{m}}\times \mathbb{Z}_{p^{n}}$.

(v) If $G_i\cong M_p(n+1,1,1)$, then $G_i\Y N\cong M_p(n+1,1,1)\Y\mathbb{Z}_{p^m}$.

(vi) If $G_i\cong M_p(w,1,1)$, where $1\leq w<n+1$, then $G_i\Y N\cong M_p(m+1,1)\times \mathbb{Z}_{p^n}$ or
$M_p(1,1,1)\Y \mathbb{Z}_{p^{m}}\times \mathbb{Z}_{p^{n}}$.

\end{lemma}

\begin{proof}
(i) - (iv) may be obtained similar to  (i) -(iv) of  Lemma \ref{CPC1}.

(v)
Suppose that
$$G_i=\langle x_i,y_i ,c\ | \ x_i^{p^{n+1}}=y_i^p=c^p=1, [x_i,y_i]=c,[x_i,c]=[y_i,c]=1 \rangle \cong M_p(n+1,1,1).$$
 Let $ x_i^p=z_1^{l}z_2^{j}$, where
$1\leq l\leq p^n, 1\leq j\leq p^m$.
If both $l$ and $j$ can be divisible by $p$, then the order of $z_1^{l}z_2^{j}$ is less than $p^{n}$, which is a contradiction.
If $l$ can be divisible by $p$, then $(j,p)=1$ and $x_i^{p^n}=z_1^{p^{n-1}l}z_2^{p^{n-1}j}=z_2^{p^{n-1}j}$, which is impossible.
Thus  $(l,p)=1$. It follows that
$$G_i\Y N=\langle x_i, y_i,z_2\rangle\cong M_p(n+1,1,1)\Y \mathbb{Z}_{p^m}. $$

(vi)
Suppose that
$$G_i=\langle x_i,y_i ,c\ | \ x_i^{p^{w}}=y_i^p=c^p=1, [x_i,y_i]=c,[x_i,c]=[y_i,c]=1 \rangle \cong M_p(w,1,1),$$
 where $1\leq w<n+1$. Let $ x_i^p=z_1^{l}z_2^{j}$, where
$1\leq l\leq p^n, 1\leq j\leq p^m$.
 $l$ can be divisible by $p$ since $1=x_i^{p^w}=z_1^{p^{w-1}l}z_2^{p^{w-1}j}$. Let $l=pl_1$.

If $(j,p)=1$, then $G_i\Y N=\langle x_iz_1^{-l_1}, y_i,z_1\rangle\cong M_p(m+1,1)\times \mathbb{Z}_{p^n}.$

Suppose $j$ can be divisible by $p$. Let $j=pj_1$. Hence
$$G_i\Y N=\langle x_iz^{-l_1}z_2^{-j_1}, y_i, z_2, z_1\rangle\cong M_p(1,1,1)\Y \mathbb{Z}_{p^m}\times\mathbb{Z}_{p^n}.$$ $\hfill\square$

\end{proof}

According to Lemmas \ref{STG} and \ref{CPC1},  $(G_i\Y N)'s$ have the
following isomorphism types: $M_p(n+1,m+1)$, $M_p(n+1,1)\times \mathbb{Z}_{p^m}$, $M_p(m+1,1,1)\Y \mathbb{Z}_{p^n}$
 and $M_p(1,1,1)\Y \mathbb{Z}_{p^n}\times\mathbb{Z}_{p^m}$. Hence
 the central product of arbitrary two factors $G_1$ and $G_2$ may only be considered  among each other of
$M_p(n+1,m+1)$, $M_p(n+1,1)$, $M_p(m+1,1,1)$ and $M_p(1,1,1)$.

\begin{lemma} \label{GGC1}
Suppose $n\geq m\geq 1$ and $c\in \langle z_1\rangle$. Let $G_1$ and $G_2$ be arbitrary two factors which are isomorphic to  $M_p(n+1,m+1), M_p(n+1,1)$  or $ M_p(m+1,1,1)$.

(i) If $G_1$ and $G_2$ are isomorphic to $M_p(n+1,m+1)$, then $G_1\Y G_2\cong
          M_p(n+1,1)\Y M_p(m+1,1,1) $ or $M_p(n+1,m+1)\Y M_p(1,1,1)$.

(ii) If $G_1$ is isomorphic to $M_p(n+1,m+1)$ and $G_2$ is isomorphic to $M_p(n+1,1)$ or $M_p(m+1,1,1)$, then $G_1\Y G_2\cong M_p(n+1,m+1)\Y M_p(1,1,1)$.

(iii)  If both $G_1$ and $G_2$ are isomorphic to $M_p(m+1,1,1)$, then $G_1\Y G_2\cong M_p(m+1,1,1)\Y M_p(1,1,1)$ or $M_p(m+1,1,1)\Y M_p(m+1,1)$( if $n=m$).

(iv) If both $G_1$ and $G_2$ are isomorphic to $M_p(n+1,1)$, then $G_1\Y G_2\cong  M_p(n+1,1)\Y M_p(1,1,1)$ or $M_p(n+1,1)\Y M_p(m+1,1,1)$.

\end{lemma}

\begin{proof}
(i) Suppose that
$$G_i=\langle x_i,y_i\ | \ x_i^{p^{n+1}}=y_i^{p^{m+1}}=1, x_i^{y_i}=x_i^{1+p^n} \rangle\cong M_p(n+1,m+1),$$
where  $i=1,2$. By adjusting parameters,   suppose $x_i^{p^n}=c$, $i=1,2$.
Obviously $\zeta G_1=\zeta G_2=N$. Assume that $x_2^p=x_1^{pl}y_1^{pj}$ and
$y_2^p=x_1^{ps}y_1^{pt}$. If $t$ is divisible by $p$, then $1\neq y_2^{p^m}=x_1^{p^ms}y_1^{p^mt}=x_1^{p^ms}$, which is a contradiction.
From this,  we may obtain $(t, p)=1$.  Thus there exists an integer number $t_1$ such that $tt_1\equiv 1\pmod{p}$.

Assume  that   $t_1\equiv -1\pmod{p}$. Obviously  $t\equiv -1\pmod{p}$.
 If  $n>m$, then  $(l,p)=1$ and $s$ is divisible by $p$. Let $s=ps_1$.
Since  $c=x_2^{p^n}=x_1^{p^nl}y_1^{p^nj}=c^ly_1^{p^nj}$, we have $c^{1-l}=y_1^{p^nj}$, that is,
$l\equiv 1\pmod{p}$.  Since
$$[x_2^{-1}x_1^ly_1^j, y_2^{-1}x_1^sy_1^t]=[x_2,y_2][x_1,y_1]^{lt}=cc^{-1}=1,$$
 we have $$G_1\Y G_2=\langle  x_2^{-1}x_1^ly_1^j,  y_1\rangle\Y\langle x_2, y_2^{-1}x_1^sy_1^t \rangle\cong M_p(m+1,1,1)\Y M_p(n+1,1).$$

Suppose  $n=m$. Obviously,  $l$ and $j$ can not simultaneously be divisible by $p$,  nor can $s$ and $t$.
If $j$ and $p$ are  coprime, then $c= x_2^{p^m}=x_1^{p^ml}y_1^{p^mj}=c^ly_1^{p^mj}$, which
 is impossible. Hence $j$ is divisible by $p$ and $l \equiv 1\pmod{p}$. Let $j=pj_1$.
Obviously,  $t$ are coprime to $p$.  We similarly have
 $$G_1\Y G_2=\langle y_2,x_1^{l}y_1^{pj_1}x_2^{-1}\rangle\Y\langle x_1 , x_1^sy_1^{t}y_2^{-1}\rangle\cong M_p(m+1,1,1)\Y M_p(m+1,1).$$

We next suppose  $t_1+1$ is not divisible by  $p$. Obviously
 $(x_2x_1^{t_1})^{p^n}\neq 1$ and $(x_2x_1^{t_1})^{p^{n+1}}=1$, that is,
 $x_2x_1^{t_1}$ is of order $p^{n+1}$. Hence
 $$G_1\Y G_2=\langle  x_2x_1^{t_1},  y_2\rangle\Y\langle x_1, y_2^{-1}x_1^sy_1^t \rangle\cong M_p(n+1,m+1)\Y M_p(n+1,1).$$
 Utilizing the following $(ii)$ in advance,  we may obtain $G_1\Y G_2\cong M_p(n+1,m+1)\Y M_p(1,1,1).$

(ii) Suppose that
$$G_1=\langle x_1,y_1 \ | \ x_1^{p^{n+1}}=y_1^{p^{m+1}}=1, x_1^{y_1}=x_1^{1+p^n} \rangle\cong M_p(n+1,m+1),$$
$$G_2=\langle x_2,y_2\ | \ x_2^{p^{n+1}}=y_2^{p}=1, x_2^{y_2}=x_2^{1+p^n} \rangle\cong M_p(n+1,1)$$
or $$G_2=\langle x_2,y_2,c\ | \ x_2^{p^{m+1}}=y_2^{p}=c^p=1,[x_2,y_2]=c,[x_2,c]=[y_2,c]=1 \rangle\cong M_p(m+1,1,1).$$
For each of the isomorphism types of $G_2$,  obviously $\zeta G_1=N\geq \zeta G_2\geq\langle x_2^p\rangle$.
 Let $x_2^p=x_1^{pl}y_1^{pj}$.
Note that $$[x_2x_1^{-l}y_1^{-j},x_1y_2^{-j}]=[x_2,y_2^{-j}][y_1^{-j},x_1]=c^{-j}c^j=1$$
and $$[x_2x_1^{-l}y_1^{-j},y_1y_2^l]=[x_2,y_2^l][x_1^{-l},y_1]=c^lc^{-l}=1.$$
 Thus
$$G_1\Y G_2=\langle x_1y_2^{-j},y_1y_2^l \rangle\Y\langle x_2x_1^{-l}y_1^{-j}, y_2 \rangle\cong M_p(n+1,m+1)\Y M_p(1,1,1).$$

(iii) Suppose that
$$G_i=\langle x_i,y_i,c\ | \ x_i^{p^{m+1}}=y_i^{p}=c^p=1,[x_i,y_i]=c,[x_i,c]=[y_i,c]=1 \rangle\cong M_p(m+1,1,1),$$
where  $i=1,2$. Let $x_1^p=z_1^{l}z_2^{j}$. If $j$ is divisible by $p$, then $x_1^{p^m}=z_1^{p^{m-1}l}z_2^{p^{m-1}j}=z_1^{p^{m-1}l}$,
which is a contradiction. Thus $(j,p)=1$. Note that $\langle x_1^p\rangle\times\langle z_1\rangle=N=\langle x_2^p\rangle\times\langle z_1\rangle$.
Let $x_2^p=x_1^{ps}z_1^t$. Obviously, $(s,p)=1$.

If $t$ is divisible by $p$, then we may let $t=pt_1$ and  replace $x_1$ by $x_1^{s}z_1^{t_1}$. At this time,
 $ x_1^p=x_2^p$. From this, we have
$$G_1\Y G_2=\langle x_1, y_1y_2 \rangle\Y\langle x_2x_1^{-1},y_2 \rangle\cong M_p(m+1,1,1)\Y M_p(1,1,1).$$
If $(t,p)=1$, then $n=m$ and
$$G_1\Y G_2=\langle x_1, y_1y_2^s \rangle\Y\langle x_2x_1^{-s},y_2 \rangle\cong M_p(m+1,1,1)\Y M_p(m+1,1).$$

(iv) Suppose that
$$G_i=\langle x_i, y_i\ | \ x_i^{p^{n+1}}=y_i^{p}=1,x_i^{y_i}=x_i^{1+p^n} \rangle\cong M_p(n+1,1),$$
where $i=1,2$. Obviously, $\langle x_1^p\rangle\times\langle z_2\rangle=N=\langle x_2^p\rangle\times\langle z_2\rangle$.
Let $x_2^p=x_1^{ps}z_2^t$, where  $(s,p)=1$.

If $t$ is divisible by $p$, then we may let $t=pt_1$ and  replace $x_1$ by $x_1^{s}z_2^{t_1}$. At this time,
 $ x_1^p=x_2^p$. From this, we have
$$G_1\Y G_2=\langle x_1, y_1y_2 \rangle\Y\langle x_2x_1^{-1},y_2 \rangle\cong M_p(n+1,1)\Y M_p(1,1,1).$$
If $(t,p)=1$, then
$$G_1\Y G_2=\langle x_1, y_1y_2^s \rangle\Y\langle x_2x_1^{-s},y_2 \rangle\cong M_p(n+1,1)\Y M_p(m+1,1,1).$$$\hfill\square$

\end{proof}

\begin{theorem} \label{ST1}
Suppose that $n\geq m\geq 1$ and $c\in \langle z_1\rangle$. Then the isomorphism types of $G$ are as follows.

(i) $M_p(n+1,m+1)\Y \underbrace{M_p(1,1,1)\Y\cdots\Y M_p(1,1,1)}_{k-1}\times\underbrace{\mathbb{Z}_p\times\cdots\times\mathbb{Z}_p}_{r-2}$, $k\geq 1$;

(ii) $M_p(n+1,1)\Y M_p(m+1,1,1) \Y\underbrace{M_p(1,1,1)\Y\cdots\Y M_p(1,1,1)}_{k-2} \times\underbrace{\mathbb{Z}_p\times\cdots\times\mathbb{Z}_p}_{r-2}$, $k\geq 2$;

(iii) $M_p(n+1,1)\Y \underbrace{M_p(1,1,1)\Y\cdots\Y M_p(1,1,1)}_{k-1}\times \mathbb{Z}_{p^{m}}\times\underbrace{\mathbb{Z}_p\times\cdots\times\mathbb{Z}_p}_{r-2}$, $k\geq 1$;

(iv) $M_p(m+1,1,1)\Y\underbrace{M_p(1,1,1)\Y\cdots\Y M_p(1,1,1)}_{k-1}\Y \mathbb{Z}_{p^{n}}\times\underbrace{\mathbb{Z}_p\times\cdots\times\mathbb{Z}_p}_{r-2}$, $k\geq 1$;

(v) $\underbrace{M_p(1,1,1)\Y\cdots\Y M_p(1,1,1)}_{k}\Y \mathbb{Z}_{p^{n}}\times \mathbb{Z}_{p^{m}}
\times\underbrace{\mathbb{Z}_p\times\cdots\times\mathbb{Z}_p}_{r-2}$, $k\geq 1$.
\end{theorem}

\begin{proof}
We rewrite $G$ by $(G_1\Y N)\Y (G_2\Y N)\Y\cdots\Y (G_k\Y N)\Y A_1$. According to Lemmas \ref{STG} and \ref{CPC1},
the isomorphism types of $( G_i\Y N)'s$ are $M_p(n+1,m+1)$, $M_p(n+1,1)\times \mathbb{Z}_{p^m}$, $M_p(m+1,1,1)\Y \mathbb{Z}_{p^n}$ and $M_p(1,1,1)\Y \mathbb{Z}_{p^n} \times \mathbb{Z}_{p^m}$.

Suppose that $n>m$.
If $\mathrm{Exp}(E)$ is equal to $p^{n+1}$, then there exists some $G_i$ such that $\mathrm{Exp}(G_i)=p^{n+1}$. Without loss of generality,
we assume that $\mathrm{Exp}(G_1)=p^{n+1}$. Hence $G_1\Y N\cong M_p(n+1,m+1)$ or $M_p(n+1,1)\times \mathbb{Z}_{p^{m}}$ by Lemma \ref{CPC1}.
If there exists some $G_i\Y N$, which  is isomorphic to $ M_p(n+1,m+1)$, then
we may obtain the isomorphism types (i) and (ii) of $G$ by Lemma \ref{GGC1}.
Assume that every one of  $(G_i\Y N)'s$ is not isomorphic to $ M_p(n+1,m+1)$.  Since  $M_p(n+1,1)\Y M_p(m+1,1,1)$
will produce $M_p(n+1,m+1)$, avoiding  repeating the above, we may  obtain  the isomorphism type (iii) by (iv) in Lemma \ref{GGC1}.

Assume that $\mathrm{Exp}(E)$ is less than $p^{n+1}$. At this case, the isomorphism types of $(G_i\Y N)'s$ are $M_p(m+1,1,1)\Y \mathbb{Z}_{p^n}$ and $M_p(1,1,1)\Y \mathbb{Z}_{p^n} \times \mathbb{Z}_{p^m}$.
If there exists some $G_i\Y N$ which is isomorphic to  $M_p(m+1,1,1)\Y \mathbb{Z}_{p^n}$, then we may obtain isomorphism type (iv) of $G$ according to (iii) of Lemma \ref{GGC1}.  Otherwise, (v)  may be obtained.

We next assume that $n=m$. If $\mathrm{Exp}(E)$ is equal to $p^{m+1}$, then we may
suppose $\mathrm{Exp}(G_1)=p^{m+1}$. Hence $G_1\Y N\cong M_p(m+1,m+1)$, $M_p(m+1,1,1)\Y \mathbb{Z}_{p^{m}}$ or  $M_p(m+1,1)\times \mathbb{Z}_{p^m}$ by Lemma \ref{CPC1}.
If  there exists some  $G_i\Y N$, which is isomorphic to $ M_p(m+1,m+1)$, then
we may obtain the isomorphism types (i) and (ii) of $G$ by Lemma \ref{GGC1}.
 Assume that every one of  $(G_i\Y N)'s$ is not isomorphic to $ M_p(m+1,m+1)$.  Obviously, $M_p(m+1,1,1)$ and $M_p(m+1,1)$ can not appear simultaneously.
If $M_p(m+1,1)$  comes into being and $ M_p(m+1,m+1)$ can not be produced,  then we may obtain the isomorphism type (iii) of $G$. If $M_p(m+1,1,1)$  appears and  $ M_p(m+1,m+1)$ can not be produced,
then (iv) may be obtained.
Assume  that $\mathrm{Exp}(E)$ is less than $p^{m+1}$.  At this case, (v) may be obtained. $\hfill\square$

\end{proof}

Avoiding  repeating the above, according to Lemma \ref{CPC2}, we only consider the case $n>m\geq 1$ and the central products among $M_p(m+1,n+1), M_p(m+1,1), M_p(n+1,1,1)$ and $ M_p(1,1,1)$.

\begin{lemma} \label{GGC2}
Suppose that $n> m\geq 1$ and $c\in \langle z_2\rangle$. Let $G_1$ and $G_2$ be arbitrary two factors which are isomorphic to $M_p(m+1,n+1), M_p(m+1,1)$ or $ M_p(n+1,1,1)$.

(i) If both $G_1$ and $G_2$ are isomorphic to $M_p(m+1,n+1)$, then
$G_1\Y G_2\cong M_p(m+1,n+1)\Y M_p(1,1,1) \ or\  M_p(m+1,1)\Y M_p(n+1,1,1)$.

(ii) If $G_1$ is  isomorphic to $M_p(m+1,n+1)$ and $G_2$ is isomorphic to $M_p(m+1,1)$ or $M_p(n+1,1,1)$, then $G_1\Y G_2\cong M_p(m+1,n+1)\Y M_p(1,1,1)$.

(iii)  If both $G_1$ and $G_2$ are isomorphic to $M_p(n+1,1,1)$, then $G_1\Y G_2\cong M_p(n+1,1,1)\Y M_p(1,1,1)$ or $ M_p(m+1,1)\Y M_p(n+1,1,1)$.

(iv) If both $G_1$ and $G_2$ are isomorphic to $M_p(m+1,1)$, then $G_1\Y G_2\cong  M_p(m+1,1)\Y M_p(1,1,1)$.

\end{lemma}

\begin{proof}
(i) Suppose that
$$G_i=\langle x_i,y_i \ | \ x_i^{p^{m+1}}=y_i^{p^{n+1}}=1, x_i^{y_i}=x_i^{1+p^m} \rangle\cong M_p(m+1,n+1),$$
where  $i=1,2$.
Obviously,  $\zeta G_1=\zeta G_2=N$. Hence we may assume that $x_2^p=x_1^{pl}y_1^{pj}$ and
$y_2^p=x_1^{ps}y_1^{pt}$. From this,  $j$ is divisible by $p$ and let $j=pj_1$.  If $l$ is divisible by $p$, then $c= x_2^{p^m}=x_1^{p^ml}y_1^{p^mj}=y_1^{p^mj}$, which is a contradiction.
Thus  $(l,p)=1$.
If $t$ is divisible by $p$, then $1\neq y_2^{p^n}=x_1^{p^ns}y_1^{p^nt}=1$ since $n>m\geq 1$, which is a contradiction.
It follows that $t$ is coprime to $p$.

If $y_1^{p^n}\neq y_2^{-p^nl}$, then $y_1y_2^l$ has  order $p^{n+1}$.
Since
$$[x_1^ly_1^jx_2^{-1}, y_1y_2^l]=[x_1,y_1]^{l}[x_2,y_2]^{-l}=c^lc^{-l}=1, $$
we have
$$G_1\Y G_2=\langle x_1,y_1y_2^l\rangle\Y \langle y_2,x_1^ly_1^jx_2^{-1}\rangle
\cong M_p(m+1,n+1)\Y M_p(n+1,1,1),$$
By (ii) in advance,  we may obtain  $G_1\Y G_2\cong M_p(m+1,n+1)\Y M_p(1,1,1).$

Assume that  $y_1^{p^n}= y_2^{-p^nl}$.
Since
$y_2^{-p^nl}=x_1^{-p^nsl}y_1^{-p^nlt}=y_1^{p^n},$
$lt\equiv -1\pmod{p}.$ From this, we have
$$[x_1^ly_1^jx_2^{-1},x_1^sy_1^ty_2^{-1}]=[x_1,y_1]^{lt}[x_2,y_2]=c^{lt}c=1.$$
Hence
$$G_1\Y G_2=\langle x_1,x_1^sy_1^ty_2^{-1}\rangle\Y \langle y_2,x_1^ly_1^jx_2^{-1}\rangle
\cong M_p(m+1,1)\Y M_p(n+1,1,1).$$

(ii) Suppose that
$$G_1=\langle x_1,y_1 \ | \ x_1^{p^{m+1}}=y_1^{p^{n+1}}=1, x_1^{y_1}=x_1^{1+p^n} \rangle\cong M_p(m+1,n+1),$$
 $$G_2=\langle x_2,y_2\ | \ x_2^{p^{m+1}}=y_2^{p}=1, x_2^{y_2}=x_2^{1+p^m} \rangle\cong M_p(m+1,1)$$
or  $$G_2=\langle x_2,y_2,c\ | \ x_2^{p^{n+1}}=y_2^{p}=c^p=1,[x_2,y_2]=c,[x_2,c]=[y_2,c]=1 \rangle\cong M_p(n+1,1,1).$$
For two isomorphism types of $G_2$,   $\zeta G_1=N\geq \zeta G_2\geq\langle x_2^p\rangle$. Hence we may let $x_2^p=x_1^{pl}y_1^{pj}$.

For the first type of  $G_{2}$,  we have $p|j$ and
$$G_1\Y G_2=\langle x_1, y_1y_2^l \rangle\Y\langle x_2x_1^{-l}y_1^{-j}, y_2 \rangle\cong M_p(m+1,n+1)\Y M_p(1,1,1).$$

For the second type of $G_2$,  we have $(j,p)=1$ and
$$G_1\Y G_2=\langle x_1y_2^{-j}, y_1y_2^l \rangle\Y\langle x_2x_1^{-l}y_1^{-j}, y_2 \rangle\cong M_p(m+1,n+1)\Y M_p(1,1,1).$$

(iii) Suppose that
$$G_i=\langle x_i,y_i,c\ | \ x_i^{p^{n+1}}=y_i^{p}=c^p=1,[x_i,y_i]=c,[x_i,c]=[y_i,c]=1 \rangle\cong M_p(n+1,1,1),$$
where  $i=1, 2$.  Let $x_1^p=z_1^{l}z_2^{j}$.  Obviously, $(l, p)=1$  and
 $\langle x_1^p\rangle\times\langle z_2\rangle=N=\langle x_2^p\rangle\times\langle z_2\rangle$.
Let $x_2^p=x_1^{ps}z_2^t$. Obviously, $(s, p)=1$.

If $t$ is divisible by $p$, then we may let $t=pt_1$ and  replace $x_1$ by $x_1^{s}z_1^{t_1}$. At this time,
 $ x_1^p=x_2^p$. From this, we have
$$G_1\Y G_2=\langle x_1, y_1y_2 \rangle\Y\langle x_2x_1^{-1},y_2 \rangle\cong M_p(n+1,1,1)\Y M_p(1,1,1).$$
If $(t,p)=1$, then
$$G_1\Y G_2=\langle x_1, y_1y_2^s \rangle\Y\langle x_2x_1^{-s},y_2 \rangle\cong M_p(n+1,1,1)\Y M_p(m+1,1).$$

(iv) Suppose that
$$G_i=\langle x_i, y_i\ | \ x_i^{p^{m+1}}=y_i^{p}=1,x_i^{y_i}=x_i^{1+p^m} \rangle\cong M_p(m+1,1),$$
where $i=1,2$. Since  $\langle x_1^p\rangle\times\langle z_1\rangle=N=\langle x_2^p\rangle\times\langle z_1\rangle$,
we may let $x_2^p=x_1^{ps}z_1^t$.
Note that $c=x_2^{p^n}=x_1^{p^ns}z_1^{p^{n-1}t}=c^sz_1^{p^{n-1}t}$. Hence  $s\equiv 1\pmod{p}$ and $p|t$.
Let $t=pt_1$ and  replace $x_1$ by $x_1^{s}z_1^{t_1}$. At this time,
 $ x_1^p=x_2^p$. From this, we have
$$G_1\Y G_2=\langle x_1, y_1y_2 \rangle\Y\langle x_2x_1^{-1},y_2 \rangle\cong M_p(m+1,1)\Y M_p(1,1,1).$$ $\hfill\square$

\end{proof}

\begin{theorem} \label{ST2}
Suppose that $n>m\geq 1$ and $c\in \langle z_2\rangle$. Then the isomorphism types of $G$ are as follows.

(i) $M_p(m+1,n+1)\Y \underbrace{M_p(1,1,1)\Y\cdots\Y M_p(1,1,1)}_{k-1}\times\underbrace{\mathbb{Z}_p\times\cdots\times\mathbb{Z}_p}_{r-2}$, $k\geq 1$;

(ii) $M_p(m+1,1)\Y M_p(n+1,1,1) \Y\underbrace{M_p(1,1,1)\Y\cdots\Y M_p(1,1,1)}_{k-2} \times\underbrace{\mathbb{Z}_p\times\cdots\times\mathbb{Z}_p}_{r-2}$, $k\geq 2$;

(iii) $M_p(n+1,1,1)\Y \underbrace{M_p(1,1,1)\Y\cdots\Y M_p(1,1,1)}_{k-1}\Y \mathbb{Z}_{p^{m}}\times\underbrace{\mathbb{Z}_p\times\cdots\times\mathbb{Z}_p}_{r-2}$, $k\geq 1$;

(iv) $M_p(m+1,1)\Y\underbrace{M_p(1,1,1)\Y\cdots\Y M_p(1,1,1)}_{k-1}\times \mathbb{Z}_{p^{n}}\times\underbrace{\mathbb{Z}_p\times\cdots\times\mathbb{Z}_p}_{r-2}$, $k\geq 1$;

(v) $\underbrace{M_p(1,1,1)\Y\cdots\Y M_p(1,1,1)}_{k}\Y \mathbb{Z}_{p^{m}}\times \mathbb{Z}_{p^{n}}
\times\underbrace{\mathbb{Z}_p\times\cdots\times\mathbb{Z}_p}_{r-2}$, $k\geq 1$.

\end{theorem}

\begin{proof}
By Lemma \ref{STG},
$$G=E\Y A_1=G_1\Y G_2\Y\cdots\Y G_k\Y A_1=(G_1\Y N)\Y (G_2\Y N)\Y\cdots\Y (G_k\Y N)\Y A_1.$$
By Lemmas \ref{STG} and  \ref{CPC2},
  the isomorphism types of $ (G_i\Y N)'s$ are $M_p(m+1,n+1)$, $M_p(n+1,1,1)\Y \mathbb{Z}_{p^m}$, $M_p(m+1,1)\times \mathbb{Z}_{p^n}$ and $M_p(1,1,1)\Y \mathbb{Z}_{p^m} \times \mathbb{Z}_{p^n}$.

If $\mathrm{Exp}(E)=p^{n+1}$, then there exists some $G_i$ such that $\mathrm{Exp}(G_i)=p^{n+1}$. Without loss of generality,
assume that $\mathrm{Exp}(G_1)=p^{n+1}$. Hence $G_1\Y N\cong M_p(m+1,n+1)$ or $M_p(n+1,1,1)\Y \mathbb{Z}_{p^{m}}$ by Lemma \ref{CPC2}.
If  there exists some $ G_i\Y N$, which  are isomorphic to $ M_p(m+1,n+1)$, then
we may obtain the isomorphism types (i) and (ii) of $G$ by Lemma \ref{GGC2}.
Assume that every one of  $ (G_i\Y N)'s$ is not isomorphic to $ M_p(m+1,n+1)$.  Note that   $M_p(n+1,1,1)\Y M_p(m+1,1)$ will produce $M_p(m+1,n+1)$.   $M_p(n+1,1,1)$  comes into being and $ M_p(m+1,m+1)$ can not be produced in order to avoid  repeating the above, thus we may  obtain  the isomorphism type (iii) by (iii) in Lemma \ref{GGC2}.

Assume that $\mathrm{Exp}(E)$ is less than $p^{n+1}$. At this time, the isomorphism types of $(G_i\Y N)'s$ are $M_p(m+1,1)\times \mathbb{Z}_{p^n}$ and $M_p(1,1,1)\Y \mathbb{Z}_{p^m} \times \mathbb{Z}_{p^n}$.
If there exists some $G_i\Y N$ satisfying $G_i\Y N\cong M_p(m+1,1)\times \mathbb{Z}_{p^n}$, then we may obtain isomorphism type (iv) of  $G$ according to (iv) of  Lemma \ref{GGC2}.  Otherwise, (v)  may be obtained. $\hfill\square$

\end{proof}

\subsection{ The isomorphism types $A_2$  and $A_3$ of $\zeta G$}

In the section,  we first consider the isomorphic type $A_2$ of $\zeta G$. Let
$$\zeta G=\langle z_1\rangle\times\langle z_2\rangle\times\cdots\times\langle z_r\rangle\cong\mathbb{Z}_{p^{n+1}}\times \mathbb{Z}_{p^{m}}\times\mathbb{Z}_p\times\cdots\times\mathbb{Z}_p, N=\langle z_1^p\rangle\times\langle z_2\rangle,$$
 Without loss of generality, we may always  let $c\in \langle z_1^p\rangle$ or $c\in \langle z_2\rangle$.
$E$ is the central product of $G_1,G_2,\ldots, G_k$ and the isomorphism types of $(G_i)'s$  are  $M_p(u,v)$ and $M_p(w,1,1)$, where $u,v,w\leq n+1$ in (2) of Lemma \ref{STG}.

\begin{lemma} \label{CPC3}
Suppose that $n\geq m\geq 1$ and $c\in \langle z_1\rangle$.

(i) If $G_i\cong M_p(n+1,v)$, then $1\leq v\leq m+1$ and $G_i\Y  \langle z_1,z_2\rangle\cong M_p(m+1,1,1)\Y \mathbb{Z}_{p^{n+1}}$ or $M_p(1,1,1)\Y \mathbb{Z}_{p^{n+1}}\times \mathbb{Z}_{p^{m}}$.

(ii) If $n=m$ and $G_i\cong M_p(u, m+1)$, where $2\leq u<m+1$, then $G_i\Y  \langle z_1,z_2\rangle\cong M_p(m+1,1,1)\Y \mathbb{Z}_{p^{m+1}}$. If $n>m$, then $G_i$ is not isomorphic to $M_p(u,n+1)$, where $2\leq u\leq m+1$.

(iii) If $G_i\cong M_p(u,v)$, where $2\leq u<n+1$ and $1\leq v<n+1$, then $G_i\Y  \langle z_1,z_2\rangle\cong M_p(m+1,1,1)\Y \mathbb{Z}_{p^{n+1}}$ or $M_p(1,1,1)\Y \mathbb{Z}_{p^{n+1}}\times \mathbb{Z}_{p^{m}}$.

(iv) If $G_i\cong M_p(w,1,1)$, then  $1\leq w<n+1$ (If $n>m$) or $1\leq w\leq m+1$ (If $n=m$).  Further, $G_i\Y  \langle z_1,z_2\rangle\cong M_p(m+1,1,1)\Y \mathbb{Z}_{p^{n+1}}$ or $M_p(1,1,1)\Y \mathbb{Z}_{p^{n+1}}\times \mathbb{Z}_{p^{m}}$.

\end{lemma}

\begin{proof}
(i) Suppose that
$$G_i=\langle x_i,y_i \ | \ x_i^{p^{n+1}}=y_i^{p^{v}}=1, x_i^{y_i}=x_i^{1+p^n} \rangle \cong M_p(n+1,v).$$
Obviously,  $\zeta G_i=\langle x_i^p, y_i^p\rangle\leq N$. From this,  we have $p^{n+v-1}\leq |N|=p^{n+m}$ and
$v\leq m+1$. Let $ x_i^p=z_1^{pl}z_2^{j}$ and $ y_i^p=z_1^{ps}z_2^{t}$, where
$1\leq l,s\leq p^{n}, 1\leq j,t\leq p^m$.

Suppose that $v=m+1$. Obviously, $\zeta G_i=N$.
If $t$ is divisible by $p$, then  $1\neq y_i^{p^m}=z_1^{p^{m}s}z_2^{p^{m-1}t}=z_1^{p^{m}s}$, which is a contradiction.
From this, we have  $t$ and $p$ are coprime.

 If $j$ and $p$ are coprime,  then we may let $x_i^p=z_1^{pl}z_2$ and
$y_i^p=z_1^{ps}z_2$ without loss of generality. At this case,
$$G_i\Y\langle z_1,z_2\rangle=\langle y_iz_1^{-s}, x_iy_i^{-1}z_1^{s-l}\rangle\Y\langle z_1\rangle\cong M_p(m+1,1,1)\Y \mathbb{Z}_{p^{n+1}}.$$
Assume that  $j$ is divisible by $p$ and  $j=pj_1$. Hence
$$G_i\Y\langle z_1,z_2\rangle=\langle y_iz_1^{-s}, x_iz_1^{-l}z_2^{-j_1} \rangle\Y\langle z_1\rangle\cong M_p(m+1,1,1)\Y \mathbb{Z}_{p^{n+1}}.$$

We next suppose $1\leq v<m+1$. Since $1=y_i^{p^v}=z_1^{p^{v}s}z_2^{p^{v-1}t}$, we have
$t$ is divisible by  $p^{m-v+1}$. Let $t=pt_1$.
If $j$ and $p$ are coprime, then $$ G_i\Y\langle z_1,z_2\rangle=\langle x_iz_1^{-l},y_iz_1^{-s}z_2^{-t_1},z_1\rangle\cong M_p(m+1,1,1)\Y\mathbb{Z}_{p^{n+1}}.$$
Assume that $j$ is divisible by $p$ and $j=pj_1$. It follows that
 $$ G_i\Y\langle z_1,z_2\rangle=\langle x_iz_1^{-l}z_2^{-j_1},y_iz_1^{-s}z_2^{-t_1}\rangle\Y\langle z_1,z_2 \rangle\cong M_p(1,1,1)\Y\mathbb{Z}_{p^{n+1}}\times\mathbb{Z}_{p^m}.$$

(ii)  Suppose that $$ G_i=\langle x_i,y_i \ | \ x_i^{p^{u}}=y_i^{p^{n+1}}=1, x_i^{y_i}=x_i^{1+p^{u-1}} \rangle \cong M_p(u,n+1),$$
where  $2\leq u\leq m+1$. Note that $\zeta G_i=\langle x_i^p,y_i^p\rangle\leq\langle z_1^p,z_2\rangle$. At this case,  let $x_i^p=z_1^{pl}z_2^{j}$ and $y_i^p=z_1^{ps}z_2^{t}$.

If $n=m$, then we only consider the case  $2\leq u< m+1$ in order to avoid repeating (i).  Since $1= x_i^{p^u}=z_1^{p^ul}z_2^{p^{u-1}j}$,
both $l$ and $j$ are divisible by $p^{m+1-u}$.  Let $j=pj_1$.
Since $1\neq y_i^{p^m}=z_1^{p^ms}z_2^{p^{m-1}t}=c^sz_2^{p^{m-1}t}$,  we have $t$ and $p$ are coprime.
From this, we have
$$G_i\Y\langle z_1,z_2\rangle=\langle y_iz_1^{-s}, x_iz_1^{-l}z_2^{-j_1},z_1\rangle\cong M_p(m+1,1,1)\Y \mathbb{Z}_{p^{m+1}}.$$
Assume that  $n>m$. At this case,
 $1\neq y_i^{p^n}=z_1^{p^ns}z_2^{p^{n-1}t}=z_1^{p^ns}=c^s$, a contradiction.

(iii)
Suppose that
$$G_i=\langle x_i,y_i \ | \ x_i^{p^{u}}=y_i^{p^{v}}=1, x_i^{y_i}=x_i^{1+p^{u-1}} \rangle \cong M_p(u,v),$$
where $u, v\leq n$.
Let $ x_i^p=z_1^{pl}z_2^{j}$ and $ y_i^p=z_1^{ps}z_2^{t}$, where
$1\leq l,s\leq p^n, 1\leq j,t\leq p^m$.

If $(j,p)=1=(t,p)$, then  we may assume that $ x_i^p=z_1^{pl}z_2$ and $ y_i^p=z_1^{ps}z_2$ without loss of generality. At this case, we have
$$G_i\Y\langle z_1,z_2\rangle=\langle x_iz_1^{-l}, x_iy_i^{-1}z_1^{s-l}, z_1\rangle\cong M_p(m+1,1,1)\Y \mathbb{Z}_{p^{n+1}}.$$

Suppose that  $(j,p)=1$ and $p|t$.  Let $t=pt_1$. From this, we have
$$G_i\Y\langle z_1,z_2\rangle=\langle x_iz_1^{-l}, y_iz_1^{-s}z_2^{-t_1},z_1\rangle\cong M_p(m+1,1,1)\Y \mathbb{Z}_{p^{n+1}}.$$
Similarly, we may obtain the same result for the case $p|j$ and $(t,p)=1$.

Assume that  $p|j$ and $p|t$. Let $j=pj_2$ and $t=pt_2$.  From this, we have
$$G_i\Y\langle z_1,z_2\rangle=\langle x_iz_1^{-l}z_2^{-j_2}, y_iz_1^{-s}z_2^{-t_2},z_1,z_2\rangle\cong M_p(1,1,1)\Y \mathbb{Z}_{p^{n+1}}\times \mathbb{Z}_{p^m}. $$

(iv)
Suppose that
$$G_i=\langle x_i,y_i ,c\ | \ x_i^{p^{w}}=y_i^p=c^p=1, [x_i,y_i]=c,[x_i,c]=[y_i,c]=1 \rangle \cong M_p(w,1,1).$$
Note that $\langle x_i^p,c\rangle\leq \langle z_1^p,z_2\rangle$.  Let $ x_i^p=z_1^{pl}z_2^{j}$, where
$1\leq l\leq p^n, 1\leq j\leq p^m$. Obviously,  $1\leq w\leq n+1$. If $w=n+1$ and $n>m$, then
$1\neq x_i^{p^n}=z_1^{p^nl}z_2^{p^{n-1}j}=z_1^{p^nl}=c^l$, which is a contradiction. From this, we have
$1\leq w\leq n+1$ (If $ n=m$) or $1\leq w< n+1$ (If $ n>m$).

If $j$ and $p$ are coprime, then
 $$G_i\Y\langle z_1,z_2\rangle=\langle x_iz_1^{-l}, y_i,z_1\rangle\cong M_p(m+1,1,1)\Y \mathbb{Z}_{p^{n+1}}. $$
 Assume that $j$ is divisible by $p$ and  $j=pj_1$. From this, we have
 $$G_i\Y\langle z_1,z_2\rangle=\langle x_iz_1^{-l}z_2^{-j_1}, y_i,z_1,z_2\rangle\cong M_p(1,1,1)\Y \mathbb{Z}_{p^{n+1}}\times \mathbb{Z}_{p^{m}}. $$ $\hfill\square$

\end{proof}

According to  Lemmas  \ref{STG} and \ref{CPC3},  the factors  of the central product of $E$ have only  two types
$M_p(m+1,1,1), M_p(1,1,1)$.  Since  $G_1\Y G_2$ is only isomorphic to $M_p(m+1,1,1)\Y M_p(1,1,1)$ if  both $G_1$
and $G_2$ are isomorphic to $M_p(m+1,1,1)$  according to  (iii) of  Lemma \ref{GGC1},  we may obtain the following result
similar to Theorem \ref{ST1}.

\begin{theorem} \label{ST3}
Suppose that $n\geq m\geq 1$ and $c\in \langle z_1\rangle$. Then the isomorphism types of $G$ are as follows.

(i) $M_p(m+1,1,1)\Y\underbrace{M_p(1,1,1)\Y\cdots\Y M_p(1,1,1)}_{k-1}\Y \mathbb{Z}_{p^{n+1}}\times\underbrace{\mathbb{Z}_p\times\cdots\times\mathbb{Z}_p}_{r-2}$, $k\geq 1$;

(ii) $\underbrace{M_p(1,1,1)\Y\cdots\Y M_p(1,1,1)}_{k}\Y \mathbb{Z}_{p^{n+1}}\times \mathbb{Z}_{p^{m}}
\times\underbrace{\mathbb{Z}_p\times\cdots\times\mathbb{Z}_p}_{r-2}$, $k\geq 1$.
\end{theorem}

\begin{lemma} \label{CPC4}
Assume that $n\geq m\geq 1$ and $c\in \langle z_2\rangle$.

(i) If $G_i\cong M_p(u,n+1)$, then $2\leq u\leq m+1$ and $G_i\Y  \langle z_1,z_2\rangle\cong
         M_p(m+1,1)\times \mathbb{Z}_{p^{n+1}}$  or $ M_p(1,1,1)\Y\mathbb{Z}_{p^m}\times\mathbb{Z}_{p^{n+1}}. $

(ii) If $n=m$ and $G_i\cong M_p(m+1, v)$, where $1\leq v<m+1$, then $G_i\Y  \langle z_1,z_2\rangle\cong M_p(m+1,1)\times \mathbb{Z}_{p^{m+1}}$. If $n>m$, then $G_i$ is not isomorphic to $M_p(n+1,v)$, where $1\leq v\leq m+1$.

(iii) If $G_i\cong M_p(u,v)$, where $2\leq u<n+1$ and $1\leq v<n+1$, then  we may obtain the same results of $(i)$.

(iv) If $G_i\cong M_p(w,1,1)$, where $1\leq w\leq n+1$, then  $G_i\Y  \langle z_1,z_2\rangle\cong M_p(m+1,1)\times\mathbb{Z}_{p^{n+1}}$ or $D_8\Y  \mathbb{Z}_{p^{m}}\times\mathbb{Z}_{p^{n+1}}$.
\end{lemma}

\begin{proof}
(i) Suppose that
$$G_i=\langle x_i,y_i \ | \ x_i^{p^{u}}=y_i^{p^{n+1}}=1, x_i^{y_i}=x_i^{1+p^{u-1}} \rangle \cong M_p(u,n+1).$$
Obviously,   $\zeta G_i=\langle x_i^p, y_i^p\rangle\leq N$. From this, we have $p^{n+u-1}\leq |N|=p^{n+m}$ and
$2\leq u\leq m+1$. Let $ x_i^p=z_1^{pl}z_2^{j}$ and $ y_i^p=z_1^{ps}z_2^{t}$, where
$1\leq l,s\leq p^{n}, 1\leq j,t\leq p^m$.

Suppose that $u=m+1$. Obviously, $\zeta G_i=N$.
If $j$ is divisible by $p$, then  $c= x_i^{p^m}=z_1^{p^{m}l}z_2^{p^{m-1}j}=z_1^{p^{m}l}$, which is a contradiction.
Thus $j$ and $p$ are coprime.

 Assume that  $t$ is divisible by $p$ and  $t=pt_1$.  It follows that
 $$G_i\Y\langle z_1,z_2\rangle=\langle x_iz_1^{-l}, y_iz_1^{-s}z_2^{-t_1}, z_1\rangle\cong M_p(m+1,1)\times \mathbb{Z}_{p^{n+1}}.$$
 If $t$ and $p$ are coprime, then we may assume that $x_i^p=z_1^{pl}z_2$ and
$y_i^p=z_1^{ps}z_2$ without loss of generality. It follows that
$$G_i\Y\langle z_1,z_2\rangle=\langle x_iz_1^{-l}, x_iy_i^{-1}z_1^{s-l}, z_1\rangle\cong M_p(m+1,1)\times \mathbb{Z}_{p^{n+1}}.$$

We next suppose $2\leq u<m+1$. Note that $1=x_i^{p^u}=z_1^{p^{u}l}z_2^{p^{u-1}j}$. Thus
$j$ is divisible by  $p^{m-u+1}$. Let $j=pj_1$.
If $t$ and $p$ are coprime, then $$ G_i\Y\langle z_1,z_2\rangle=\langle y_iz_1^{-s}, x_iz_1^{-l}z_2^{-j_1},z_1\rangle\cong M_p(m+1,1)\times\mathbb{Z}_{p^{n+1}}.$$
Assume that $t$ is divisible by $p$ and $t=pt_2$. It follows that
 $$ G_i\Y\langle z_1,z_2\rangle=\langle x_iz_1^{-l}z_2^{-j_1},y_iz_1^{-s}z_2^{-t_2}\rangle\Y\langle z_1,z_2 \rangle\cong M_p(1,1,1)\Y\mathbb{Z}_{p^m}\times\mathbb{Z}_{p^{n+1}}.$$

(ii)  Suppose that  $$G_i=\langle x_i,y_i \ | \ x_i^{p^{n+1}}=y_i^{p^{v}}=1, x_i^{y_i}=x_i^{1+p^{n}} \rangle \cong M_p(n+1,v),$$
where  $1\leq v\leq m+1$. Obviously, $\zeta G_i=\langle x_i^p,y_i^p\rangle\leq N=\langle z_1^p,z_2\rangle$. At this case,  assume that $x_i^p=z_1^{pl}z_2^{j}$ and $y_i^p=z_1^{ps}z_2^{t}$.

If $n=m$, then we only consider the case  $1\leq v< m+1$.  Since $c= x_i^{p^m}=z_1^{p^ml}z_2^{p^{m-1}j}=z_1^{p^ml}c^j$,
$j$ and $p$ are coprime.
Since $1= y_i^{p^v}=z_1^{p^{v}s}z_2^{p^{v-1}t}$,   $t$ are divisible by $p^{m-v+1}$. From this,  we have $p|t$ and assume that $t=pt_1$.
It follows that
$$G_i\Y\langle z_1,z_2\rangle=\langle x_iz_1^{-l}, y_iz_1^{-s}z_2^{-t_1},z_1\rangle\cong M_p(m+1,1)\times \mathbb{Z}_{p^{m+1}}.$$
If  $n>m$, then $c= x_i^{p^n}=z_1^{p^nl}z_2^{p^{n-1}j}=z_1^{p^nl}$, which is a contradiction.

(iii)
Suppose that
$$G_i=\langle x_i,y_i \ | \ x_i^{p^{u}}=y_i^{p^{v}}=1, x_i^{y_i}=x_i^{1+p^{u-1}} \rangle \cong M_p(u,v),$$
where $u,v\leq n$.
 Let $ x_i^p=z_1^{pl}z_2^{j}$ and $ y_i^p=z_1^{ps}z_2^{t}$, where
$1\leq l,s\leq p^n, 1\leq j,t\leq p^m$.

Assume that $(j,p)=1$ and $p|t$. Let $t=pt_1$. Hence
$$G_i\Y\langle z_1,z_2\rangle=\langle x_iz_1^{-l}, y_iz_1^{-s}z_2^{-t_1},z_1\rangle\cong M_p(m+1,1)\times \mathbb{Z}_{p^{n+1}}.$$
Similarly, we may obtain the same result for the case $p|j$ and $(t,p)=1$.

Assume that  $p|j$ and $p|t$. Let $j=pj_2$ and $t=pt_2$. Hence
$$G_i\Y\langle z_1,z_2\rangle=\langle x_iz_1^{-l}z_2^{-j_2}, y_iz_1^{-s}z_2^{-t_2},z_1,z_2\rangle\cong M_p(1,1,1)\Y\mathbb{Z}_{p^m}\times \mathbb{Z}_{p^{n+1}}. $$

If $(j,p)=1=(t,p)$, then  we may let $ x_i^p=z_1^{pl}z_2$ and $ y_i^p=z_1^{ps}z_2$ without loss of generality. At this case,
$$G_i\Y
\langle z_1,z_2\rangle=\langle x_iz_1^{-l}, x_iy_i^{-1}z_1^{s-l}, z_1\rangle\cong M_p(m+1,1)\times \mathbb{Z}_{p^{n+1}}.$$

(iv)
Suppose that
$$G_i=\langle x_i,y_i ,c\ | \ x_i^{p^{w}}=y_i^p=c^p=1, [x_i,y_i]=c,[x_i,c]=[y_i,c]=1 \rangle \cong M_p(w,1,1).$$
Since $\langle x_i^p,c\rangle\leq \langle z_1^p,z_2\rangle$,  we may let $ x_i^p=z_1^{pl}z_2^{j}$, where
$1\leq l\leq p^n, 1\leq j\leq p^m$.
If $j$ and $p$ are coprime, then
 $$G_i\Y\langle z_1,z_2\rangle=\langle x_iz_1^{-l}, y_i,z_1\rangle\cong M_p(m+1,1)\times \mathbb{Z}_{p^{n+1}}. $$
 Assume that $j$ is divisible by $p$ and  $j=pj_1$. It follows that
 $$G_i\Y
 \langle z_1,z_2\rangle=\langle x_iz_1^{-l}z_2^{-j_1}, y_i,z_1,z_2\rangle\cong M_p(1,1,1)\Y \mathbb{Z}_{p^{m}}\times\mathbb{Z}_{p^{n+1}}. $$$\hfill\square$

\end{proof}

By Lemmas  \ref{STG} and \ref{CPC4},  the factors  of the central product of $E$ have only  two types
$M_p(m+1,1), M_p(1,1,1)$.
Since  $G_1\Y G_2$ is only isomorphic to $M_p(m+1,1)\Y M_p(1,1,1)$ if  both $G_1$
and $G_2$ are isomorphic to $M_p(m+1,1)$,  we may obtain the following results.

\begin{theorem} \label{ST4}
Suppose that $n\geq m\geq 1$ and $c\in \langle z_2\rangle$. Then the isomorphism types of $G$ are as follows.

(i) $M_p(m+1,1)\Y\underbrace{M_p(1,1,1)\Y\cdots\Y M_p(1,1,1)}_{k-1}\times \mathbb{Z}_{p^{n+1}}\times\underbrace{\mathbb{Z}_p\times\cdots\times\mathbb{Z}_p}_{r-2}$, $k\geq 1$;

(ii) $\underbrace{M_p(1,1,1)\Y\cdots\Y M_p(1,1,1)}_{k}\Y  \mathbb{Z}_{p^{m}}\times\mathbb{Z}_{p^{n+1}}
\times\underbrace{\mathbb{Z}_p\times\cdots\times\mathbb{Z}_p}_{r-2}$, $k\geq 1$.

\end{theorem}

Suppose that
$N=\langle z_1\rangle\times\langle z_2^p\rangle.$
 Similar to the case $A_2$, we have the following the results for the case $A_3$.

\begin{corollary} \label{ST5}
Suppose that $n\geq m\geq 1$ and $c\in \langle z_1\rangle$. Then the isomorphism types of $G$ are as follows.

(i) $M_p(n+1,1)\Y\underbrace{M_p(1,1,1)\Y\cdots\Y M_p(1,1,1)}_{k-1}\times \mathbb{Z}_{p^{m+1}}\times\underbrace{\mathbb{Z}_p\times\cdots\times\mathbb{Z}_p}_{r-2}$, $k\geq 1$;

(ii) $\underbrace{M_p(1,1,1)\Y\cdots\Y M_p(1,1,1)}_{k}\Y \mathbb{Z}_{p^{n}}\times \mathbb{Z}_{p^{m+1}}
\times\underbrace{\mathbb{Z}_p\times\cdots\times\mathbb{Z}_p}_{r-2}$, $k\geq 1$.

\end{corollary}

\begin{corollary} \label{ST6}
Suppose that $n\geq m\geq 1$ and $c\in \langle z_2\rangle$. Then the isomorphism types of $G$ are as follows.

(i) $M_p(n+1,1,1)\Y\underbrace{M_p(1,1,1)\Y\cdots\Y M_p(1,1,1)}_{k-1}\Y \mathbb{Z}_{p^{m+1}}\times\underbrace{\mathbb{Z}_p\times\cdots\times\mathbb{Z}_p}_{r-2}$, $k\geq 1$;

(ii) $\underbrace{M_p(1,1,1)\Y\cdots\Y M_p(1,1,1)}_{k}\Y  \mathbb{Z}_{p^{m+1}}\times\mathbb{Z}_{p^{n}}
\times\underbrace{\mathbb{Z}_p\times\cdots\times\mathbb{Z}_p}_{r-2}$, $k\geq 1$.
\end{corollary}

\subsection{ The isomorphism type $A_4$ of $\zeta G$}

In the section,
$$\zeta G=\langle z_1\rangle\times\langle z_2\rangle\times\cdots\times\langle z_r\rangle\cong\mathbb{Z}_{p^{n+1}}\times \mathbb{Z}_{p^{m+1}}\times\mathbb{Z}_p\times\cdots\times\mathbb{Z}_p, N=\langle z_1^p\rangle\times\langle z_2^p\rangle.$$
 Without loss of generality, we may always  let $c\in \langle z_1^p\rangle$ or $c\in \langle z_2^p\rangle$.

\begin{lemma} \label{CPC7}
Suppose $n\geq m\geq 1$ and $c\in \langle z_1\rangle$ or $\langle z_2\rangle$.

(i) If $G_i\cong M_p(u,v)$, where $2\leq u\leq n+1$ and $1\leq v\leq n+1$, then $G_i\Y  \langle z_1,z_2\rangle\cong M_p(1,1,1)\Y (\mathbb{Z}_{p^{n+1}}\times \mathbb{Z}_{p^{m+1}})$.

(ii) If $G_i\cong M_p(w,1,1)$, where $1\leq w\leq n+1$,  then $G_i\Y  \langle z_1,z_2\rangle\cong M_p(1,1,1)\Y (\mathbb{Z}_{p^{n+1}}\times \mathbb{Z}_{p^{m+1}})$.
\end{lemma}

\begin{proof}
(i) Suppose that
$$G_i=\langle x_i,y_i \ | \ x_i^{p^{u}}=y_i^{p^{v}}=1, x_i^{y_i}=x_i^{1+p^{u-1}} \rangle \cong M_p(u,v).$$
Since $\zeta G_i=\langle x_i^p, y_i^p\rangle\leq N$,   $ x_i^p=z_1^{pl}z_2^{pj}$ and $ y_i^p=z_1^{ps}z_2^{pt}$, where
$1\leq l,s\leq p^{n}, 1\leq j,t\leq p^m$.  It follows that
 $$G_i\Y\langle z_1,z_2\rangle=\langle x_iz_1^{-l}z_2^{-j},y_iz_1^{-s}z_2^{-t},z_1,z_2\rangle\cong M_p(1,1,1)\Y (\mathbb{Z}_{p^{n+1}}\times \mathbb{Z}_{p^{m+1}}).$$

(ii)
Suppose that
$$G_i=\langle x_i,y_i ,c\ | \ x_i^{p^{w}}=y_i^p=c^p=1, [x_i,y_i]=c,[x_i,c]=[y_i,c]=1 \rangle \cong M_p(w,1,1).$$
where $1\leq w\leq n+1$.
 Let $ x_i^p=z_1^{pl}z_2^{pj}$, where
$1\leq l\leq p^n, 1\leq j\leq p^m$. Hence
 $$G_i\Y\langle z_1,z_2\rangle=\langle x_iz_1^{-l}z_2^{-j},y_i,z_1,z_2\rangle\cong M_p(1,1,1)\Y (\mathbb{Z}_{p^{n+1}}\times \mathbb{Z}_{p^{m+1}}).$$ $\hfill\square$

\end{proof}

According to Lemmas \ref{STG} and \ref{CPC7},  it is easy to obtain the following result.

\begin{theorem} \label{ST7}
Suppose that $n\geq m\geq 1$ and $c\in \langle z_1\rangle$ or $\langle z_2\rangle$. Then the isomorphism type of $G$ is
$$\underbrace{M_p(1,1,1)\Y\cdots\Y M_p(1,1,1)}_{k}\Y (\mathbb{Z}_{p^{n+1}}\times \mathbb{Z}_{p^{m+1}})
\times\underbrace{\mathbb{Z}_2\times\cdots\times\mathbb{Z}_2}_{r-2},  k\geq 1.$$

\end{theorem}

\section{ The normalized unit groups}

Suppose that $G$ is a  finite $p$-group given by a
central extension of the form
$$1\longrightarrow N=\mathbb{Z}_{p^n}\times \mathbb{Z}_{p^m} \longrightarrow G \longrightarrow \mathbb{Z}_p\times
\cdots\times \mathbb{Z}_p \longrightarrow 1$$ and $G'=\langle c\rangle\cong
\mathbb{Z}_p$, $n\geq m\geq 1$ and $p>2$.  According to Lemma \ref{STG}, we still assume that $x_1\zeta G,\ldots, x_{k}\zeta G, y_1\zeta G,\ldots, y_{k}\zeta G$ are the generated elements of  $G/\zeta G$, $z_1,\ldots, z_r$ are the generated elements of $\zeta G$ for the convenience.

\begin{lemma} \label{Conclass} The conjugacy class of every noncentral element of $G$ contains $p$ elements.
\end{lemma}

\begin{proof}
If $g\notin \zeta G$, then there exists an element $h$ such that $[g,h]\neq 1$.
We may suppose that
$$g=\prod\limits_{i=1}^{k}x_i^{a_{1i}}y_i^{b_{1i}}\prod\limits_{j=1}^{r}z_j^{c_{1j}},
h=\prod\limits_{i=1}^{k}x_i^{a_{2i}}y_i^{b_{2i}}\prod\limits_{j=1}^{r}z_j^{c_{2j}},$$
then
\begin{eqnarray*}
\begin{split}
g^h&=\prod\limits_{i=1}^{k}\left(x_{i}^{a_{1i}}\cdot y_{i}^{b_{1i}}\right)^{x_{i}^{a_{2i}}\cdot y_{i}^{b_{2i}}}\prod\limits_{j=1}^{r}z_j^{c_j}\\
&=\prod\limits_{i=1}^{k}\left(x_{i}^{a_{1i}}\right)^{y_{i}^{b_{2i}}}\cdot\left(y_{i}^{b_{1i}}\right)^{x_{i}^{a_{2i}}}\prod\limits_{j=1}^{r}z_j^{c_j}\\
&=\prod\limits_{i=1}^{k}\left(x_{i}^{a_{1i}}\cdot c^{a_{1i}b_{2i}}\right)\cdot\left(y_{i}^{b_{1i}}\cdot c^{-b_{1i}a_{2i}}\right)\prod\limits_{j=1}^{r}z_j^{c_j}\\
&=gc^{s}
\end{split}
\end{eqnarray*}
where $s=\sum\limits_{i=1}^k(a_{1i}b_{2i}-b_{1i}a_{2i})$.
From this,  the conjugacy class containing $g$ is $\{g, gc, \ldots, gc^{p-1}\}$.$\hfill\qedsymbol$
\end{proof}

According to  Lemma \ref{Conclass}, we know that the number of the conjugacy classes of $G$ which contain $p$ elements is $(|G|-|\zeta G|)/p$ and is denoted by $t$.
Suppose that these conjugacy classes are $C_{g_1},C_{g_2},\ldots,C_{g_t}$. Let $C:=\langle 1+\widehat{C}_{g_i}\ |\ i=1,2, \ldots t\rangle$ and $I:=\{\sum\limits_{i=1}^ta_i\widehat{C}_{g_i}\ |\ a_i\in \mathbb{F}_p,i=1,2,\ldots,t\}$.

\begin{lemma} \label{Centernu}  $\zeta(V(\mathbb{F}_pG))=V(\mathbb{F}_p\zeta G)\times C$ and $|\zeta(V(\mathbb{F}_pG))|=p^{\frac{|G|+(p-1)|\zeta G|-p}{p}}$.
\end{lemma}

\begin{proof} By Lemma \ref{Conclass}, we have that $\widehat{C}_{g_i}=g_i\widehat{G'}$, $i=1,2,\ldots,t$.
Since $\widehat{G'}^p=0$,
 $\widehat{C}_{g_i}\widehat{C}_{g_j}=0$ for any $ i,j\in \{1,2,\ldots,t\}$.
From this,  $I$  is an ideal of $\zeta(\mathbb{F}_pG)$. Note that  all conjugacy classes of $G$
form a basis for $\zeta(\mathbb{F}_pG)$. Hence,  for any $a\in \zeta(V(\mathbb{F}_pG))$, there exist $b\in\mathbb{F}_p \zeta G$, $a_i\in \mathbb{F}_p$($i=1,2,\ldots,t$) such that
$a=b+\sum\limits_{i=1}^ta_i\widehat{C}_{g_i}$. Obviously, $1=\varepsilon(a)=\varepsilon(b)$, which  implies that $b\in V(\mathbb{F}_p \zeta G)$.
From this,  there exist $b_i\in \mathbb{F}_p$($i=1,2,\ldots,t$) such that
$$a=b\left(1+b^{-1}\sum\limits_{i=1}^ta_i\widehat{C}_{g_i}\right)=b\left(1+\sum\limits_{i=1}^tb_i\widehat{C}_{g_i}\right)=b\prod\limits_{i=1}^t(1+\widehat{C}_{g_i})^{b_i}.$$
Obviously $\langle 1+\widehat{C}_{g_i}\ |\ i=1,2, \ldots t\rangle$ is an elementary abelian $p$-group with order $p^t$.
The structure of $\zeta(V(\mathbb{F}_pG))$  is desired. Since $V(\mathbb{F}_p\zeta G))$ is of order $p^{|\zeta G|-1}$,
 $$|\zeta(V(\mathbb{F}_pG))|=p^{|\zeta G|-1+t}=p^{\frac{|G|+(p-1)|\zeta G|-p}{p}}.$$ $\hfill\qedsymbol$
\end{proof}

\begin{theorem} \label{ExpV}
$V(\mathbb{F}_pG)^{p^l}=V(\mathbb{F}_pG^{p^l})$, where $l\geq 2$.
 \end{theorem}

\begin{proof}
Assume that $\alpha=\sum\limits_{g\in G}\alpha_gg$
is the element of $V(\mathbb{F}_pG)$.  By Lemma \ref{Brauer},
$\alpha^p=\sum\limits_{g\in G}\alpha_gg^p+\delta$, where $\delta\in [\mathbb{F}_pG,\mathbb{F}_pG]$. Since $\alpha^p$ and $g^p$ are elements of $\zeta(V(\mathbb{F}_pG))$,
$\delta$ is an element of $\zeta(\mathbb{F}_pG)$. Also, by  Lemma  \ref{Brauer}, $\delta$ is a linear combination of the $\widehat{C_{g_i}}$. From this,  we have  $\delta^p=0$. It follows that
$$\alpha^{p^2}=\sum\limits_{g\in G}\alpha_gg^{p^2}+\delta^p=\sum\limits_{g\in G}\alpha_gg^{p^2}.$$
Hence $V(\mathbb{F}_pG)^{p^l}=V(\mathbb{F}_pG^{p^l})$, where $l\geq 2$.
$\hfill\qedsymbol$
\end{proof}

\begin{theorem}\label{Vp}
$V(\mathbb{F}_pG)^{p}C=V(\mathbb{F}_pG^{p})\times C$.
\end{theorem}

\begin{proof}
For any $\alpha\in V(\mathbb{F}_pG)$, suppose that $\alpha=\sum\limits_{g\in G}\alpha_gg$.
According to Lemma \ref{Brauer}, there exists an element $\gamma\in [\mathbb{F}_pG,\mathbb{F}_pG]$
such that $\alpha^p=\sum\limits_{g\in G}\alpha_gg^p+\gamma$. Note that $\sum\limits_{g\in G}\alpha_gg^p$
is an element of $V(\mathbb{F}_pG^{p})$ and $V(\mathbb{F}_pG^{p})\leq \zeta(\mathbb{F}_pG)$.
Since $I$  is an ideal of $\zeta(\mathbb{F}_pG)$, there exists  $\gamma_1\in C$ such that
$$\alpha^p=(\sum\limits_{g\in G}\alpha_gg^p)(1+(\sum\limits_{g\in G}\alpha_gg^p)^{-1}\gamma)=(\sum\limits_{g\in G}\alpha_gg^p)\gamma_1\in V(\mathbb{F}_pG^{p})\times C. $$
Hence $V(\mathbb{F}_pG)^{p}C\leq V(\mathbb{F}_pG^{p})\times C$ and vice versa.$\hfill\qedsymbol$
\end{proof}

\begin{remark}
According to Theorem \ref{StrG},  the isomorphism types (1.5) and (1.10) of  $G$ have exponent $p$  when  $m=n=1$.
At this time, $V(\mathbb{F}_pG)^{p}>V(\mathbb{F}_pG^p)=1$ since $(x_1+y_1-1)^p=(x_1+y_1)^p-1\neq 1$ by Lemma \ref{abp}.
Further $V(\mathbb{F}_pG)^{p}=C$ by Lemma 5 in \cite{BB1}.
\end{remark}

\begin{corollary}
$G\bigcap V(\mathbb{F}_pG)^{p}=G^p$.
\end{corollary}

\begin{proof}
For any element $x\in G\bigcap (V(\mathbb{F}_pG^{p})\times C)$, we suppose $x=yz$, where $y\in V(\mathbb{F}_pG^{p})$
and $z\in C$. By Lemma \ref{Vpc}, $x$ is included in the center of $V(\mathbb{F}_pG)$, that is, $x\in \zeta G$.
From this, $y^{-1}x$ is included in $V(\mathbb{F}_p\zeta G)$, that is, $z$ is included in the intersection of $V(\mathbb{F}_p\zeta G)$
and $C$. By Lemma \ref{Centernu},  $z$ is equal to 1.  It follows that
$$G\bigcap (V(\mathbb{F}_pG^{p})\times C)=G\bigcap V(\mathbb{F}_pG^{p})= G^p.$$
Since $$G^p\leq G\bigcap V(\mathbb{F}_pG)^{p}\leq G\bigcap V(\mathbb{F}_pG)^{p}C, $$
 the result is true by Theorem \ref{Vp}.
$\hfill\qedsymbol$
\end{proof}

\begin{theorem}\label{O1V}
$\Omega_1(V(\mathbb{F}_pG))\leq 1+\triangle(G,\Omega_1(G))$ and
$\Omega_l(V(\mathbb{F}_pG))= 1+\triangle(G,\Omega_l(G))$, where $l\geq 2.$
\end{theorem}

\begin{proof}
Suppose that $\alpha=\sum\limits_{g\in G}\alpha_gg\in \Omega_l(V(\mathbb{F}_pG))$.
If $l=1$,  according to Lemma  \ref{Brauer}, there exists $\beta\in [\mathbb{F}_pG,\mathbb{F}_pG]$ such that
$\alpha^p=\sum\limits_{g\in G}\alpha_gg^p+\beta.$
By Lemma  \ref{Vpc}, we have
$$\sum\limits_{g\in G}\alpha_gg^p=1-\beta\in V(\mathbb{F}_p\zeta G)\bigcap C=1,$$
that is, $\sum\limits_{g\in G}\alpha_gg^p=1$ and $\beta=0$. If $l\geq 2$, then we obviously $1=\alpha^{p^i}=\sum\limits_{g\in G}\alpha_gg^{p^i}.$

Let $T$ be a transversal of $\Omega_l(G)$ in $G$. Then every element $\alpha\in \mathbb{F}_pG$
can be written as a finite sum $\alpha=\sum\limits_{i,j}r_{ij}q_ih_j$, where $q_1=1$ and $q_i\in T$, $h_j\in \Omega_l(G)$, $r_{ij}\in \mathbb{F}_p$. For $\alpha\in \Omega_l(V(\mathbb{F}_pG))$, according to the above result, we have
$$1=\alpha^{p^l}=\sum\limits_{i,j}r_{ij}q_i^{p^l}h_j^{p^l}=\sum\limits_{i}(\sum\limits_{j}r_{ij})q_i^{p^l}.$$
Hence $$0=\sum\limits_{i}(\sum\limits_{j}r_{ij})q_i^{p^l}-1
=\sum\limits_{i}(\sum\limits_{j}r_{ij})q_i^{p^l}-\sum\limits_{i}(\sum\limits_{j}r_{ij})
=\sum\limits_{i}(\sum\limits_{j}r_{ij})(q_i^{p^l}-1).$$
From this,
$\sum\limits_{j}r_{1j}=1$, $\sum\limits_{j}r_{ij}=0$ if $i\neq 1$.
It follows that
\begin{eqnarray*}
\begin{split}
\alpha &=\sum\limits_{j}r_{1j}h_j+\sum\limits_{i\neq 1}(\sum\limits_{j}r_{ij}h_j)q_i \\
 &=1+\sum\limits_{j}r_{1j}(h_j-1)+\sum\limits_{i\neq 1}(\sum\limits_{j}r_{ij}h_j)q_i-\sum\limits_{i\neq 1}(\sum\limits_{j}r_{ij})q_i \\
 &=1+\sum\limits_{i,j}r_{ij}q_i(h_j-1)
\end{split}
\end{eqnarray*}
By Lemma \ref{subbas}, we have $\alpha\in 1+ \triangle(G,\Omega_l(G)).$
It follows that $\Omega_l(V(\mathbb{F}_pG))\leq 1+\triangle(G,\Omega_l(G))$.

Conversely, for any element $\alpha\in 1+\triangle(G,\Omega_l(G))$, suppose that
$\alpha=1+\sum\limits_{i,j}r_{ij}q_i(h_j-1)$, where $q_i\in T$, $h_j\in \Omega_l(G)$ and $r_{ij}\in \mathbb{F}_p$.
We first suppose that $l=1$. For all isomorphism types of $G$ except for the types  $(1.1)$, $(1.5)$,  $(1.6)$, $(1.10)$, $(1.12)$,
$(1.14)$, $(1.16)$,  $(1.18)$ and $(1.19)$,
there exist two elements $x$ and $y$ such that $x\in T$, $y\in \Omega_1(G)$  and $[x,y]\neq 1$. At this time,
$1+x(y-1)\in 1+\triangle(G,\Omega_1(G))$ and $(1+x(y-1))^p\neq 1$ by Lemma \ref{abp}.
For the case $(1.1)$,  there exist $a, b$ satisfying $|a|=p^{n+1}, |b|=p^{m+1}$  and $[a,b]\neq 1$,
then by Lemma \ref{abp}
$$1+a(a^{p^n}b^{p^m}-1)+b(b^{p^m}-1)\in 1+\triangle(G,\Omega_1(G))$$ and
$$(1+a(a^{p^n}b^{p^m}-1)+b(b^{p^m}-1))^p\neq 1.$$
For the case $(1.6)$, we may obtain the similar result.
For the other types, there exist $a_1, b_1\in \Omega_1(G)$ such that $[a_1, b_1]\neq 1$. At this time,
$1+(a_1-1)+(b_1-1)\in 1+\triangle(G,\Omega_1(G))$ and $(a_1+b_1-1)^p\neq 1$.
In a word, $\Omega_1(V(\mathbb{F}_pG))$ is a proper subgroup of $1+\triangle(G,\Omega_1(G))$.

Suppose that $l>1$, then,  by Theorem \ref{ExpV}, we have
$$\alpha^{p^l}=1+\sum\limits_{i,j}r_{ij}q_i^{p^l}(h_j^{p^l}-1)=1.$$
Hence $1+\triangle(G,\Omega_l(G))\leq \Omega_l(V(\mathbb{F}_pG))$, which is desired.
$\hfill\qedsymbol$
\end{proof}

\begin{theorem}\label{CG1}
$\zeta \Omega_l(V(\mathbb{F}_pG))=\Omega_l(\zeta V(\mathbb{F}_pG))=\Omega_l(V(\mathbb{F}_p\zeta G))\times C$, where $l\geq 1$.
\end{theorem}

\begin{proof}
If $l\geq \mathrm{Exp}(V(\mathbb{F}_pG)$, then the result  is trivial by Lemma 4.2. We next suppose $l<\mathrm{Exp}(V(\mathbb{F}_pG)$.
According to Lemma \ref{Centernu}, we have $\Omega_l(\zeta V(\mathbb{F}_pG))=\Omega_l(V(\mathbb{F}_p\zeta G))\times C$.
Hence we only prove the first equation. Obviously, $\Omega_l(\zeta V(\mathbb{F}_pG))\leq \zeta\Omega_l(V(\mathbb{F}_pG))$.

Suppose that $\alpha\in\zeta \Omega_l(V(\mathbb{F}_pG))$.
Let $T$ be a transversal of $\Omega_l(G)$ in $G$. According to Theorem \ref{O1V},
$\alpha$ may be written as a finite sum $1+\sum\limits_{i,j}r_{ij}q_i(h_j-1)$, where $q_i\in T$, $h_j\in \Omega_l(G)$, $r_{ij}\in \mathbb{F}_p$.
For any $\beta\in \Omega_l(G)-\zeta\Omega_l(G)$,  we have $\alpha=\alpha^{\beta}$ since $\beta\in \Omega_l(V(\mathbb{F}_pG))$.
According to Theorem \ref{StrG}, all $q_i's$ may be taken from the product of the subgroups with exponent more than $p^l$ which are factors in the central product or
direct product. Hence each of $q_i's$ and $\beta$ are commutative, that is,
 $\sum\limits_{i,j}r_{ij}q_i(h_j-1)=\sum\limits_{i,j}r_{ij}q_i(h_j^{\beta}-1)$.
From this, every element  has the same coefficient which is conjugated with $h_j$ in $\Omega_l(G)$.
Without loss of  generality, we still use $j$ to distinguish from the different conjugacy class of $\Omega_l(G)$,
denote  by $\widehat{h_j}$ and $n_j$ the sum and number of the elements in the conjugacy class including in $h_j$, respectively.
Moreover, according to Lemma \ref{Conclass},  $n_j=p$ if $h_j\notin \zeta\Omega_l(G)$; $n_j=1$ if $h_j\in \zeta\Omega_l(G)$.
Hence $\alpha$ is written as
\begin{center}
$\alpha=1+\sum\limits_{\substack{  i,j\\
                                        h_j\in  \zeta\Omega_l(G)}}r_{ij}q_i(h_j-1)+\sum\limits_{\substack{  i,j\\
                                        h_j\notin  \zeta\Omega_l(G)}}r_{ij}q_i\widehat{h_j}. $
\end{center}

If $q_i$ is not included in $\zeta G$, then there exists $\rho\in G$ such that $[q_i,\rho]=c$ and $[\rho,\Omega_l(G)]=1$.
Since $1+\rho(c-1)$ is included in $\Omega_1(V(\mathbb{F}_pG))$, we have
$1+\rho(c-1)$ and $\alpha$  are commutative. Hence
$$\sum\limits_{\substack{ i,j\\
                                        h_j\in  \zeta\Omega_l(G)}}r_{ij}q_i(h_j-1)(c-1)=\sum\limits_{\substack{ i,j\\
                                        h_j\in  \zeta\Omega_l(G)}}r_{ij}q_i^{\rho}(h_j-1)(c-1).$$
According to Lemma \ref{subbas},  every element  has the same coefficient which is conjugated with $q_i$ for $h_j\in  \zeta\Omega_l(G)$.
From this, there exists $\gamma\in I$ such that  \begin{center}
$\alpha=1+\sum\limits_{\substack{  i,j\\
                                         q_i\in \zeta G\\
                                        h_j\in  \zeta\Omega_l(G)}}r_{ij}q_i(h_j-1)+\gamma.$
\end{center}

Suppose that there exists some $h_j\in \zeta\Omega_l(G)-\zeta G$, then there is $\theta\in  G$
such that $[h_j,\theta]=c$. Note that $1+\theta(c-1)\in \Omega_1(V(\mathbb{F}_pG))$.
Hence $1+\theta(c-1)$ and $\alpha$  are commutative.
From this,  every element  has the same coefficient which is conjugated with $h_j$ in $G$.
Note that $\zeta\Omega_l(G)\bigcap\zeta G=\Omega_l( \zeta G)$.
It follows  there exists $\gamma_1\in I$ such that
\begin{center}
$\alpha=1+\sum\limits_{\substack{  i,j\\
                        q_i\in \zeta G\\
                        h_j\in \Omega_l( \zeta G)}}r_{ij}q_i(h_j-1)+\gamma_1. $
\end{center}
Hence
$$\alpha\in (1+\triangle(\zeta G,\Omega_l(\zeta G))) \times C=\Omega_l(V(\mathbb{F}_p\zeta G))\times C.$$
$\hfill\qedsymbol$
\end{proof}


\begin{thebibliography}{99}

\bibitem{BB1} Z. Balogh, A. Bovdi, Group algebras with unit group of class $p$, Publ. Math. Debrecen, 2004, 65(3-4): 261-268.

\bibitem{Blackburn} S. R. Blackburn, Groups of prime power order with derived subgroup of prime order, J. Algebra, 1999, 219: 625-657.

\bibitem{Bornand} D. Bornand,  Elementary abelian subgroups in $p$-groups of class 2, Lausanne:
$\mathrm{\acute{E}cole}$ Polytechnique $\mathrm{F\acute{e}d\acute{e}rale}$  de Lausanne, 2009.

\bibitem{Bovdi} A. Bovdi, J. Kurdics, Lie  properties of the group algebra and the nilpotency class of the group of units, J. Algebra, 1999, 212(1): 28-64.

\bibitem{Bovdi2} A. Bovdi, Z. Patay, Units in group algebras of finite $p$-groups,
J. Math. Sci., 2003, 116(1): 2887-2893.


\bibitem{Winter} D. Winter,  The automorphism group of an extraspecial $p$-group, Rocky Mountain J. Math., 1972, 2: 159-168.

\bibitem{Brauer} R. Brauer,  Zur darstellungstheorie der gruppen endlicher ordnung, Math. Z., 1956, 63: 406-444.


\bibitem{Dietz} J. Dietz, Automorphisms of $p$-group given as cyclic-by elementary abelian central extensions, J. Algebra, 2201, 242: 417-432.


\bibitem{Gupta} N. Gupta, F. Levin, On the Lie ideals of ring, J. Algebra, 1983, 81: 225-231.

\bibitem{Hall} P. Hall, The calssification of prime-power groups, J. Reine Angew. Math. 1940, 182: 130-141.

\bibitem{Johnson} D. L. Johnson, The modular group ring of a finite $p$-group, Proc. Amer. Math. Soc., 1978, 68(1): 19-22.

\bibitem{LW} H. G. Liu, Y. L. Wang, The automorphism group of a generalized extraspecial $p$-gorup, Sci. China Math., 2010, 53(2): 315-334.

\bibitem{LW2} H. G. Liu, Y. L. Wang, Automorphism groups of some finite $p$-groups, Algebra Colloq., 2016, 23(4): 623-650.

\bibitem{Milies} C. P. Milies, S. K. Sehgal, An introduction to group rings, Dordrecht-Boston-London: Kluwer Academic Publishers, 2002.

\bibitem{Robinson} D. J. S. Robinson, A course in the theory of groups(Second
edition), New York: Springer-Verlag, 1996.


\bibitem{Sandling} R. Sandling,  Units in the modular group algebra of a finite abelian $p$-group, J. Pure Appl. Algebra, 1984, 33: 337-346.

\bibitem{Sandling2} R. Sandling,  The modular group algebra of a central-elementary-by-abelian $p$-group, Arch. Math., 1989, 52: 22-27.


\bibitem{WL2} Y. L. Wang, H. G. Liu, The generalization of the theorems of Winter and Dietz, Chin. Ann. Math. Ser. A, 2012, 33A(5): 609-630. (in Chinese)

\bibitem{WL3} Y. L. Wang, H. G. Liu, The unitary subgroups of group algebras of a class of finite $2$-groups with derived subgroup of order $2$, Sci. China Math. 2023, 66, https:// doi.org/10.1007/s11425-016-5135-4.

\bibitem{XuQ}  M. Y. Xu, H. P. Qu,  Finite $p$-groups,  Beijing:  Peking Unibersity Press,  2010.




\end{thebibliography}
\end{document}